%% file: tssos_siam.tex
\documentclass[onefignum,onetabnum]{siamart190516}
\usepackage{amsmath}
\usepackage{amssymb}
\usepackage{enumerate}
\usepackage{mathrsfs}
\usepackage{cases}
\usepackage{color}
\usepackage{tikz}
\usepackage{makecell}
\usepackage{multirow}
\usepackage{siunitx}
\usepackage{graphicx}
\definecolor{darkblue}{rgb}{0.0, 0.0, 0.55}
\definecolor{bordeaux}{rgb}{0.34, 0.01, 0.1}
\newsiamthm{example}{Example}

\def\Z{{\mathbb{Z}}}

\def\R{{\mathbb{R}}}

\def\N{{\mathbb{N}}}

\def\x{{\mathbf{x}}}

\def\y{{\mathbf{y}}}
\def\bu{{\mathbf{u}}}

\def\be{{\mathbf{e}}}
\def\bo{{\boldsymbol{\omega}}}

\def\a{{\boldsymbol{\alpha}}}

\def\b{{\boldsymbol{\beta}}}
\def\g{{\boldsymbol{\gamma}}}
\def\bv{{\boldsymbol{v}}}
\newcommand{\br}{\mathbf{r}}
\newcommand{\bs}{\mathbf{s}}
\def\A{{\mathscr{A}}}
\def\B{{\mathscr{B}}}
\def\S{{\mathscr{S}}}
\def\up{{\boldsymbol{\upsilon}}}

\def\supp{\hbox{\rm{supp}}}

\def\Tr{\hbox{\rm{Tr}}}

\def\int{\hbox{\rm{int}}}
\def\New{\hbox{\rm{New}}}
\def\Conv{\hbox{\rm{conv}}}

\def\DSOS{\hbox{\rm{DSOS}}}
\def\SDSOS{\hbox{\rm{SDSOS}}}
\newcommand{\newtssos}[1]{{{\color{black}#1}}}
\newcommand{\vge}{\mathbin{\rotatebox[origin=c]{90}{$\ge$}}}

\input{tssos_share}

\ifpdf
\hypersetup{
  pdftitle={TSSOS},
  pdfauthor={J. Wang, V. Magron and J.-B. Lasserre}
}
\fi

\begin{document}

\maketitle

\begin{abstract}
This paper is concerned with polynomial optimization problems.
We show how to exploit term (or monomial) sparsity of the input polynomials to obtain
a new converging hierarchy of semidefinite programming relaxations. The novelty (and distinguishing feature) of
such relaxations is to involve block-diagonal matrices obtained
in an iterative procedure performing completion of the connected components of certain adjacency graphs. The graphs are related to the terms arising in the original data and
{\em not} to the links
between variables. Our theoretical framework is then applied to compute lower bounds for polynomial optimization problems either randomly generated or coming from the networked system literature.
\end{abstract}



\section{Introduction}
In this paper we provide a new method to handle a certain class of {\em sparse} polynomial optimization problems. Roughly speaking, for problems in this class the terms (monomials) appearing in the involved polynomials satisfy a certain ``sparsity pattern" which is represented by block-diagonal binary matrices. This sparsity
pattern concerned with the structure of {\em monomials} involved
in the problem, is different from the correlative sparsity pattern already studied in \cite{waki} and related to the links between {\em variables}.

\subsection*{Background} The problem of minimizing a polynomial over a set defined by a finite conjunction of polynomial inequalities (also known as a \textit{basic semialgebraic} set), is known to be NP-hard~\cite{Lau09b}.
The \textit{moment-sum of squares} (moment-SOS) hierarchy by Lasserre~\cite{las1} is a nowadays established methodology allowing one to handle this problem.
Optimizing a polynomial can be reformulated either with a primal infinite-dimensional linear program (LP) over probability measures or with its dual LP over nonnegative polynomials.
In a nutshell, the moment-SOS hierarchy is based on the fact that one can consider a sequence of finite-dimensional primal-dual relaxations for the two above-mentioned LPs.
At each step of the hierarchy, one only needs to solve a single semidefinite program (SDP).
Under mild assumptions (slightly stronger than compactness), the related sequence of optimal values converges to the optimal value of the initial problem.
One well-known limitation of this methodology is that the size of the matrices involved in the primal-dual SDP at the $d$-th step of the hierarchy is proportional to $\binom{n+d}{n}$, where $n$ is the number of variables of the initial problem.

There are several existing ways to overcome these scalability limitations.
To compute the SOS decomposition of a given nonnegative polynomial, one can systematically reduce the size of the corresponding SDP matrix by removing the terms (monomials) which cannot appear in the support of the decomposition~\cite{re}.
One can also exploit (i) the sparsity pattern satisfied by the variables of the initial problem~\cite{Las06,waki} (see also the related SparsePOP solver~\cite{WakiKKMS08}) as well as (ii) the symmetries~\cite{Riener13} of the problem.
In particular, sparsity has been successively exploited for specific applications, e.g. for solving optimal power flow problems~\cite{Josz16}, roundoff error bound analysis \cite{toms17,toms18}, or more recently to approximate the volume of sparse semialgebraic sets \cite{Tacchi19}.
The polynomials involved in these applications have a specific \textit{correlative sparsity pattern}.
Sparse polynomial optimization is based on re-indexing the SDP matrices involved in the moment-SOS hierarchy, by considering subsets $I_1,\dots,I_p \subseteq \{1,\dots,n\}$ of the input variables.
One then obtains a sparse variant of the moment-SOS hierarchy with quasi block-diagonal SDP matrices, each block having a size related to the cardinality of these subsets.
Hence if the cardinalities are small with respect to $n$, then the resulting SDP relaxations yield significant (sometimes drastic)
computational savings. Under mild assumptions, global convergence of this sparse version of the moment-SOS hierarchy is guaranteed if the so-called \textit{running intersection property} (RIP) holds.
Recently, this methodology has been extended in~\cite{ncsparse} to sparse problems with non-commuting variables (for instance matrices).
Other SOS-based representations include the bounded degree sum of squares~\cite{TohLasserre} with its sparse variant \cite{we}.
These two latter hierarchies come with same convergence guarantees as the standard ones (under the same sparsity pattern assumption).
They involve SDP matrices of smaller size but come with potentially larger sets of linear constraints which may
sometimes result in ill-conditioned relaxations.

Other than exploiting sparsity from the perspective of variables, one can also exploit sparsity from the perspective of terms, such as sign-symmetries \cite{lo1} and minimal coordinate projections \cite{pe1} in the unconstrained case. More recently, \textit{cross sparsity patterns}, a new attempt in this direction introduced in \cite{wang}, apply to a wider class of polynomials. By exploiting cross sparsity patterns, a monomial basis used for constructing SOS decompositions is partitioned into blocks. If each block has a small size with respect to the size of the original monomial basis, then the corresponding SDP matrix is block-diagonal with small blocks and this might significantly improve the efficiency and the scalability.

{\em The present paper can be viewed as a comprehensive extension of the idea in \cite{wang}
to the constrained case and in a more general perspective.}


All the above-mentioned hierarchies require to solve a sequence of SDP relaxations. However in other convex programming frameworks, there exist alternative classes of positivity certificates also based on term sparsity.
This includes sums of \textit{nonnegative circuit polynomials} (SONC) and sums of \textit{arithmetic-geometric-exponential-means polynomials} (SAGE).
A circuit polynomial is a polynomial with support containing only monomial squares, except at most one term, whose exponent is a strict convex combination of the other exponents.
An AGE polynomial is a composition of weighted sums of exponentials with linear functionals of the variables, which is nonnegative and contains also at most one negative coefficient.
Existing frameworks \cite{Chandrasekaran:Shah:SAGE,Ghasemi:Marshall:GP,Iliman:deWolff:Circuits} allow one to compute sums of nonnegative circuits and sums of AGEs by relying on geometric programming and signomial programming, respectively.
In~\cite{ahmadi2014dsos}, the authors introduce alternative decompositions of nonnegative polynomials as \textit{diagonal sum of squares} (DSOS) and \textit{scaled diagonal sum of squares} (SDSOS).
Such decompositions can be computed via linear programming and second order cone programming, respectively, a potential advantage with respect to standard SOS-based decompositions.
For these frameworks based on SAGE/SONC/(DSOS)SDSOS decompositions, one can also handle constrained problems and derive a corresponding converging hierarchy of lower bounds.
However the underlying relaxations share the same drawback, namely
their implementation and the computation of resulting lower bounds are not easy in practice. Very recently, a combination of correlative sparsity and SDSOS has been proposed in \cite{mi}. This method does not provide a guarantee of convergence and, in its current state, is only applicable to the case of unconstrained polynomial optimization problems.

\subsection*{Contributions}
We provide a new sparse moment-SOS hierarchy based on term sparsity rather than correlative sparsity.
This is in deep contrast with the sparse variant of the moment-SOS hierarchy developed in~\cite{Las06,waki}.

\newtssos{$\bullet$ In Section~\ref{sec3}, we describe an iterative procedure to exploit the term sparsity in polynomials that describe the problem on hand. Each iteration consists of two steps, a {\em support-extension} operation followed by a {\em block-closure} operation on certain binary matrices. This iterative procedure is then applied to unconstrained polynomial optimization in Section~\ref{sec4} and constrained polynomial optimization in Section~\ref{sec5}. In both cases the iterative procedure leads to a converging moment-SOS hierarchy (called TSSOS hierarchy) of primal-dual relaxations involving {\em block-diagonal} SDP matrices. If the sizes of blocks are small with respect to the original SDP matrices, then the resulting SDP relaxations yield a significant computational saving.

$\bullet$ The TSSOS hierarchy (in the constrained case) depends on two parameters: the relaxation order $\hat{d}$ and the sparse order $k$ (corresponding to each iterative step), and hence allows one more level of flexibility by playing with the two parameters $\hat{d}$ and $k$. The optimal values of the TSSOS hierarchy, at fixed relaxation order $\hat{d}$, yield a non-decreasing sequence  converging to the optimal value of the dense moment-SOS hierarchy at the same relaxation order in a few steps (typically two or three). 
In the unconstrained case we prove that even at first iterative step ($k=1$), the optimal value of the corresponding SDP relaxation is already no worth than the one obtained with the SDSOS-based decompositions~\cite{ahmadi2014dsos}.

$\bullet$ We prove in Section~\ref{sec6} that the block-structure of the TSSOS hierarchy at each relaxation order converges to the block-structure determined by the sign-symmetries related to the support of the input data.
This also enables us to provide a new sparse variant of Putinar's Positivstellensatz~\cite{pu} for positive polynomials over basic compact semialgebraic sets. In this representation, the supports of all SOS polynomials are reduced according to the sign-symmetries.

$\bullet$ In Section~\ref{sec:benchs}, we compare the efficiency and scalability of the TSSOS hierarchy with existing frameworks on randomly generated examples as well as on problems arising from the networked system literature. The numerical results demonstrate that TSSOS has a significantly better performance in terms of   efficiency and scalability.
In addition, and although it is not guaranteed in theory, we observe in our numerical results that
the optimal value obtained at the first iterative step ($k=1$) of the TSSOS hierarchy is always the same as the one obtained
from the dense moment-SOS hierarchy on all tested examples, a very encouraging sign of efficiency. At last but not least, we emphasize that in
all numerical examples (except the Broyden banded function from \cite{waki}), the usual correlative sparsity pattern is dense or almost dense and so yields no or little computational savings (or cannot even be implemented).}

\newtssos{
As mentioned in Remark \ref{sec3-rm} and done in the companion paper \cite{chordaltssos}, one can replace block-closure by chordal-extension to exploit term sparsity, in order to obtain an even more sparse variant of the moment-SOS hierarchy: this is the so-called Chordal-TSSOS moment-SOS hierarchy.
In the present paper we treat general
polynomial optimization problems (POPs).
However, as often the case, some correlative sparsity 
is present in the input data (description) of large-scale POPs. Therefore a natural idea is to combine correlative sparsity with our
current TSSOS framework of term sparsity,
for solving large-scale POPs. Such an extension (called CS-TSSOS) is considered in our recent work \cite{cstssos} and is non-trivial as it requires extra care when manipulating monomials that involve variables of different cliques that appear in the correlative sparsity pattern. As a result, CS-TSSOS can handle large-scale POPs (e.g., instances of the celebrated Max-Cut and optimal power flow problems) with up to several thousands of variables.
}
%

%
\section{Notation and Preliminaries}
\subsection{Notation and SOS polynomials}
Let $\x=(x_1,\ldots,x_n)$ be a tuple of variables and $\R[\x]=\R[x_1,\ldots,x_n]$ be the ring of real $n$-variate polynomials. For a subset $\A\subseteq\N^n$, we denote by $\Conv(\A)$ the convex hull of $\A$. A polynomial $f\in\R[\x]$ can be written as $f(\x)=\sum_{\a\in\A}f_{\a}\x^{\a}$ with $f_{\a}\in\R, \x^{\a}=x_1^{\alpha_1}\cdots x_n^{\alpha_n}$. The support of $f$ is defined by $\supp(f)=\{\a\in\A\mid f_{\a}\ne0\}$, and the Newton polytope of $f$ is defined as the convex hull of $\supp(f)$, i.e., $\New(f)=\Conv(\{\a:\a\in\supp(f)\})$. We use $|\cdot|$ to denote the cardinality of a set. For $\A_1,\A_2\subseteq\N^n$, let $\A_1+\A_2:=\{\a_1+\a_2\mid\a_1\in\A_1,\a_2\in\A_2\}$.

For a nonempty finite set $\A\subseteq\N^n$, let $\mathscr{P}(\A)$ be the set of polynomials in $\R[\x]$ whose supports are contained in $\A$, i.e.\,$\mathscr{P}(\A)=\{f\in\R[\x]\mid\supp(f)\subseteq\A\}$ and let $\x^{\A}$ be the $|\A|$-dimensional column vector consisting of elements $\x^{\a},\a\in\A$ (fix any ordering on $\N^n$). For a positive integer $r$, the set of $r\times r$ symmetric matrices is denoted by $\mathbb{S}^r$ and the set of $r\times r$ positive semidefinite (PSD) matrices is denoted by $\mathbb{S}_+^r$.

Given a polynomial $f(\x)\in\R[\x]$, if there exist polynomials $f_1(\x),\ldots,f_t(\x)$ such that
\begin{equation}\label{sec2-eq1}
f(\x)=\sum_{i=1}^tf_i(\x)^2,
\end{equation}
then we say that $f(\x)$ is a \textit{sum of squares} (SOS) polynomial. Clearly, the existence of an SOS decomposition of a given polynomial provides a certificate for its global nonnegativity. For $d\in\N$, let $\N^n_d:=\{\a=(\alpha_i)\in\N^n\mid\sum_{i=1}^n\alpha_i\le d\}$ and assume that $f\in\mathscr{P}(\N^n_{2d})$. If we choose the {\em standard monomial basis} $\x^{\N^n_{d}}$, then the SOS condition \eqref{sec2-eq1} is equivalent to the existence of a PSD matrix $Q$ (which is called a \textit{Gram matrix} \cite{re2}) such that
\begin{equation}\label{sec2-eq2}
f(\x)=(\x^{\N^n_{d}})^TQ\x^{\N^n_{d}},
\end{equation}
which is formulized as a semidefinite program (SDP).

We say that a polynomial $f\in\mathscr{P}(\N^n_{2d})$ is {\em sparse} if the number of elements in its support $\A=\supp(f)$ is much smaller than the number of elements in $\N^n_{2d}$ that forms a support of fully dense polynomials in $\mathscr{P}(\N^n_{2d})$. When $f(\x)$ is a sparse polynomial in $\mathscr{P}(\N^n_{2d})$, the size of the corresponding SDP \eqref{sec2-eq2} can be reduced by computing a smaller monomial basis. In fact, the set $\N^n_{d}$ in (\ref{sec2-eq2}) can be replaced by the integer points in half of the Newton polytope of $f$, i.e.\,by
\begin{equation}\label{sec2-eq3}
\B=\frac{1}{2}\cdot\New(f)\cap\N^n\subseteq\N^n_{d}.
\end{equation}
See \cite{re} for a proof.
We refer to this as the \textit{Newton polytope method}.
There are also other methods to reduce the size of $\B$ further \cite{ko,pe}.
Throughout this paper, we will use a monomial basis, which is either the monomial basis given by the Newton polytope method in the unconstrained case or the standard monomial basis in the constrained case. For convenience, we abuse notation in the sequel and denote the monomial basis $\x^{\B}$ by the exponents $\B$.

\subsection{Moment matrices}
With $\y=(y_{\a})_{\a\in\N^n}$ being a sequence indexed by the standard monomial basis $\N^n$ of $\R[\x]$, let $L_{\y}:\R[\x]\rightarrow\R$ be the linear functional
\begin{equation*}
f=\sum_{\a}f_{\a}\x^{\a}\mapsto L_{\y}(f)=\sum_{\a}f_{\a}y_{\a}.
\end{equation*}
For a monomial basis $\B$, the {\em moment} matrix $M_{\B}(\y)$  associated with $\B$ and $\y$ is the matrix with rows and columns indexed by $\B$ such that
\begin{equation*}
M_{\B}(\y)_{\b\g}:=L_{\y}(\x^{\b}\x^{\g})=y_{\b+\g}, \quad\forall\b,\g\in\B.
\end{equation*}
If $\B$ is the standard monomial basis $\N^n_{d}$, we also denote $M_{\B}(\y)$ by $M_{d}(\y)$.

Suppose $g=\sum_{\a}g_{\a}\x^{\a}\in\R[\x]$ and let $\y=(y_{\a})_{\a\in\N^n}$ be given. For a positive integer $d$, the {\em localizing} matrix $M_{d}(g\y)$ associated with $g$ and $\y$ is the matrix with rows and columns indexed by $\N^n_{d}$ such that
\begin{equation*}
M_{d}(g\,\y)_{\b\g}:=L_{\y}(g\,\x^{\b}\x^{\g})=\sum_{\a}g_{\a}y_{\a+\b+\g}, \quad\forall\b,\g\in\N^n_{d}.
\end{equation*}

\section{Exploiting term sparsity in SOS decompositions}\label{sec3}
For a positive integer $r$, let $[r] := \{1,\ldots,r\}$. For matrices $A,B\in\mathbb{S}^r$, let $A\circ B\in\mathbb{S}^r$ denote the Hadamard, or entrywise, product of $A$ and $B$, defined by the equation $[A\circ B]_{ij} = A_{ij}B_{ij}$ and let $\langle A, B\rangle\in\R$ be the trace inner-product, defined by $\langle A, B\rangle=\Tr(A^TB)$. Let $\Z_2^{r\times r}$ ($\Z_2:=\{0,1\}$) be the set of $r\times r$ binary matrices. The support of a binary matrix $B\in\mathbb{S}^r\cap\Z_2^{r\times r}$ is the set of locations of nonzero entries, i.e.,
$$\supp(B):=\{(i,j)\in[r]\times[r]\mid B_{ij}=1\}.$$

For a binary matrix $B\in\mathbb{S}^r\cap\Z_2^{r\times r}$, we define the set of PSD matrices with sparsity pattern represented by $B$ as
\begin{equation*}
\mathbb{S}_+^r(B):=\{Q\in \mathbb{S}_+^r\mid B\circ Q=Q\}.
\end{equation*}

Let $f(\x)=\sum_{\a\in\A}f_{\a}\x^{\a}$ with $\supp(f)=\A$ and $\B$ be a monomial basis with $r=|\B|$. For any $\a\in\B+\B$, associate it with a binary matrix $A_{\a}\in\mathbb{S}^r\cap\Z_2^{r\times r}$ such that $[A_{\a}]_{\b\g}=1$ iff $\b+\g=\a$ for all $\b,\g\in\B$. Then $f(\x)$ is an SOS polynomial iff there exists $Q\in\mathbb{S}_+^r$ such that the following coefficient matching condition holds:
\begin{equation}\label{sec3-eq4}
\langle A_{\a}, Q\rangle=f_{\a} \textrm{ for all }\a\in\B+\B,
\end{equation}
where we set $f_{\a}=0$ if $\a\notin\A$. 
\newtssos{
For later use, we also define $A_{\S}:=\sum_{\a\in\S}A_{\a}$ for any subset $\S\subseteq\B+\B$.
}
For convenience, we define a block-closure operation on binary matrices as follows.
\begin{definition}
\newtssos{A relation $R\subseteq[r]\times[r]$ is called {\em transitive} if 
$(i,j),(j,k)\in R$ implies $(i,k)\in R$. 
The {\em transitive closure} of $R$, denoted by $\overline{R}$, is the smallest relation that contains $R$ and is transitive.}
For a binary matrix $B\in\mathbb{S}^r\cap\Z_2^{r\times r}$, let $R\subseteq[r]\times[r]$ be the adjacency relation of $B$, i.e.,\,$(i,j)\in R$ iff $B_{ij}=1$. Then define the {\em block-closure} $\overline{B}\in\mathbb{S}^r\cap\Z_2^{r\times r}$ as
\begin{equation*}
\overline{B}_{ij} :=
\begin{cases}
1,\quad&(i,j)\in\overline{R}, \\
0,\quad&\textrm{otherwise}.
\end{cases}
\end{equation*}
\end{definition}

For a binary matrix $B\in\mathbb{S}^r\cap\Z_2^{r\times r}$, the evaluation of $\overline{B}$ has a graphical description (assume $B_{ii}=1$ for all $i$). Suppose that $G$ is the adjacency graph of $B$. Then $\overline{B}$ is the adjacency
matrix of the graph obtained by completing the connected components of $G$ to complete subgraphs. 
\newtssos{Hence the evaluation of block-closure boils down to the computation of connected components of a graph, which can be done in linear time (in terms of the numbers of the vertices and edges of the graph)}. Note also that $\overline{B}$ is block-diagonal up to permutation, where each block corresponds to a connected component of $G$. Figure 1 is a simple example where $\overline{B}$ has two blocks of size $3$ and $1$ corresponding to the connected components of $G$: $\{1,3,4\}$ and $\{2\}$, respectively.
\begin{figure}[htbp]
\begin{center}
\begin{minipage}{0.6\linewidth}
\begin{displaymath}
B=\begin{bmatrix}
1&0&1&0\\
0&1&0&0\\
1&0&1&1\\
0&0&1&1\\
\end{bmatrix}
\qquad
\overline{B}=\begin{bmatrix}
1&0&1&1\\
0&1&0&0\\
1&0&1&1\\
1&0&1&1\\
\end{bmatrix}
\end{displaymath}
\end{minipage}
\begin{minipage}{0.35\linewidth}
\begin{tikzpicture}[every node/.style={circle, draw=blue!50, thick}]
\node (n1) at (0,1.5) {$1$};
\node (n2) at (0,0) {$3$};
\node (n3) at (1.5,0) {$4$};
\node (n4) at (2.5,0) {$2$};
\coordinate[label=above:$G:$] (G) at (-0.6,0.25);
\draw[thick] (n1)--(n2);
\draw[thick] (n2)--(n3);
\draw[dashed] (n1)--(n3);
\end{tikzpicture}
\end{minipage}
\end{center}
\caption{Block-closure and connected components}
\end{figure}
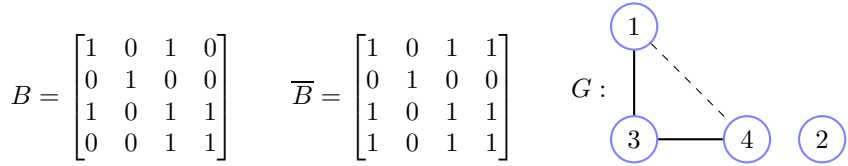

\begin{remark}\label{sec3-rm}
The block-closure operation $\overline{B}$ used in this paper can be actually replaced by a chordal-extension operation on adjacency graphs. Then take maximal cliques rather than connected components. \newtssos{See \cite{wang,chordaltssos} for more details.} 
We use the block-closure in this paper since it is very simple to determine.
\end{remark}

Let $f(\x)\in\R[\x]$ with $\supp(f)=\A$ and let $\B$ be a monomial basis with $r=|\B|$. Let $\S^{(0)}=\A\cup(2\B)$ where $2\B=\{2\b\mid\b\in\B\}$. For $k\ge1$, we recursively define binary matrices $B_{\A}^{(k)}\in\mathbb{S}^r\cap\Z_2^{r\times r}$ indexed by $\B$ via two successive steps:\\
\newtssos{
1) {\bf Support-extension}: define a binary matrix
$C_{\A}^{(k)}=A_{\S^{(k-1)}}$, i.e.,
\begin{equation*}
[C_{\A}^{(k)}]_{\b\g} :=
\begin{cases}
1,\quad&\textrm{if }\b+\g\in\S^{(k-1)}, \\
0,\quad&\textrm{otherwise}.
\end{cases}
\end{equation*}
2) {\bf Block-closure}: let $B_{\A}^{(k)}=\overline{C_{\A}^{(k)}}$
and $\S^{(k)}=\bigcup_{[B_{\A}^{(k)}]_{\b\g}=1}\{\b+\g\}$.\\
}
By construction, it is easy to see that $\supp(B_{\A}^{(k)})\subseteq\supp(B_{\A}^{(k+1)})$ for all $k\ge1$. Hence the sequence of binary matrices $(B_{\A}^{(k)})_{k\ge1}$ stabilizes after a finite number of steps. We denote the stabilized matrix by $B_{\A}^{(*)}$.

Let us denote the set of SOS polynomials supported on $\A$ by
\begin{equation*}
\Sigma(\A):=\{f\in\mathscr{P}(\A)\mid\exists Q\in \mathbb{S}_+^r\textrm{ s.t. }f=(\x^{\B})^TQ\x^{\B}\},
\end{equation*}
and for $k\ge1$, let $\Sigma_k(\A)$ be the subset of $\Sigma(\A)$ whose member admits a Gram matrix with sparsity pattern represented by $B_{\A}^{(k)}$, i.e.,
\begin{equation}\label{sec3-eq1}
\Sigma_k(\A):=\{f\in\mathscr{P}(\A)\mid\exists Q\in \mathbb{S}_+^r(B_{\A}^{(k)})\textrm{ s.t. }f=(\x^{\B})^TQ\x^{\B}\}.
\end{equation}
In addition, let
\begin{equation}\label{sec3-eq2}
\Sigma_*(\A):=\{f\in\mathscr{P}(\A)\mid\exists Q\in \mathbb{S}_+^r(B_{\A}^{(*)})\textrm{ s.t. }f=(\x^{\B})^TQ\x^{\B}\}.
\end{equation}
By construction, we have the following inclusions:
$$\Sigma_1(\A)\subseteq\Sigma_2(\A)\subseteq\cdots\subseteq\Sigma_*(\A)\subseteq\Sigma(\A).$$

\begin{theorem}\label{sec3-thm1}
For a finite set $\A\subseteq\N^n$, one has $\Sigma_*(\A)=\Sigma(\A)$.
\end{theorem}
\begin{proof}
We only need to prove the inclusion $\Sigma(\A)\subseteq\Sigma_*(\A)$.
Suppose $\B$ is a monomial basis. For any $f\in\Sigma(\A)$, let $Q\in\mathbb{S}_+^r$ be a Gram matrix of $f$ and we construct a matrix $\tilde{Q}\in\mathbb{S}_+^r$ by $\tilde{Q}=B_{\A}^{(*)}\circ Q$. \newtssos{We next show that $f=(\x^{\B})^T\tilde{Q}\x^{\B}$. Let $\S^{(*)}=\cup_{[B_{\A}^{(*)}]_{\b\g}=1}\{\b+\g\}$. By construction, $B_{\A}^{(*)}$ is stabilized under the support-extension operation and hence we have $B_{\A}^{(*)}=A_{\S^{(*)}}$. Thus $(\x^{\B})^TQ\x^{\B}-(\x^{\B})^T\tilde{Q}\x^{\B}=(\x^{\B})^T(Q-\tilde{Q})\x^{\B}=(\x^{\B})^T(A_{(\B+\B)}\circ Q-A_{\S^{(*)}}\circ Q)\x^{\B}=(\x^{\B})^T(A_{(\B+\B)\backslash\S^{(*)}}\circ Q)\x^{\B}$. Again by construction, one has $\A\subseteq\S^{(*)}$. It follows that $(\x^{\B})^T(A_{(\B+\B)\backslash\S^{(*)}}\circ Q)\x^{\B}=0$ since $(\x^{\B})^T(A_{(\B+\B)}\circ Q)\x^{\B}=f$ and $\supp(f)=\A$. Therefore, $(\x^{\B})^T\tilde{Q}\x^{\B}=(\x^{\B})^TQ\x^{\B}=f$.}

Note that $\tilde{Q}$ is block-diagonal (up to permutation) and each block of $\tilde{Q}$ is a principal submatrix of $Q$, so $\tilde{Q}$ is PSD. Thus $f\in\Sigma_{*}(\A)$.
\end{proof}

Consequently, we obtain a hierarchy of inner approximations of $\Sigma(\A)$ which reaches $\Sigma(\A)$ in a finite number of steps.

\begin{remark}
For each $k\ge1$, $Q\in\mathbb{S}_+^r(B_{\A}^{(k)})$ is block-diagonal (up to permutation). Thus checking membership in $\Sigma_k(\A)$ boils down to solving an SDP problem involving SDP matrices of small sizes if each block has a small size with respect to the original matrix. This might significantly reduce the overall computational cost.
\end{remark}

The next result states that $\Sigma_1(\A)=\Sigma(\A)$ always holds in the quadratic case.
\begin{theorem}\label{sec3-thm2}
For a finite set $\A\subseteq\N^n$, if for all $\a=(\alpha_i)\in\A$, $\sum_{i=1}^n\alpha_i\le2$, then $\Sigma_1(\A)=\Sigma(\A)$.
\end{theorem}
\begin{proof}
We only need to prove the inclusion $\Sigma(\A)\subseteq\Sigma_1(\A)$. Suppose $f\in\Sigma(\A)$ is a quadratic polynomial with $\supp(f)=\A$. Let $\B=\{\mathbf{0}\}\cup\{\be_k\}_{k=1}^n$ be the standard monomial basis and $Q=[q_{ij}]_{i,j=0}^n$ a Gram matrix of $f$. To show $f\in\Sigma_1(\A)$, it suffices prove that $Q\in\mathbb{S}_+^{n+1}(C_{\A}^{(1)})$, which holds if $[C_{\A}^{(1)}]_{ij}=0$ implies $q_{ij}=0$ for all $i,j$. Clearly, $[C_{\A}^{(1)}]_{00}=0$ implies $q_{00}=0$.  If $i=0,j>0$, from $[C_{\A}^{(1)}]_{0j}=0$ one has $\be_j\notin\A$. If $i>0,j=0$, from $[C_{\A}^{(1)}]_{i0}=0$ one has $\be_i\notin\A$. If $i,j>0$, from $[C_{\A}^{(1)}]_{ij}=0$ one has $\be_i+\be_j\notin\A$. In any of these three cases, one has $q_{ij}=0$ as desired.
\end{proof}

\section{A block SDP hierarchy for unconstrained POPs}\label{sec4}
In this section, we consider the unconstrained polynomial optimization problem:
\begin{equation*}
(\textrm{P}) : \quad \theta^* := \inf_{\x}\{f(\x) : \x\in\R^n\}
\end{equation*}
with $f(\x)\in\R[\x]$, and exploit the sparse SOS decompositions in Section\,\ref{sec3} to establish a block SDP hierarchy for $(\textrm{P})$.

Obviously, (P) is equivalent to
\begin{equation*}
(\textrm{P'}) : \quad \theta^* = \sup_{\lambda}\{\lambda\mid f(\x)-\lambda\ge0\}.
\end{equation*}
Replacing the nonnegativity condition by the stronger SOS condition, we obtain an SOS relaxation of (P):
\begin{equation*}
(\textrm{SOS}) : \quad \theta_{sos} := \sup_{\lambda}\{\lambda\mid f(\x)-\lambda\in\Sigma(\A)\},
\end{equation*}
with $\A=\{\mathbf{0}\}\cup\supp(f)$.
If $f$ is sparse and we replace the nonnegativity condition in (P') by the sparse SOS conditions \eqref{sec3-eq1}, then we obtain a hierarchy of sparse SOS relaxations of (P):
\begin{equation}\label{sec4-eq1}
(\textrm{P}^k)^* : \quad \theta_{k} := \sup_{\lambda}\{\lambda\mid f(\x)-\lambda\in\Sigma_k(\A)\}, \quad k=1,2,\ldots.
\end{equation}
For each $k$, $(\textrm{P}^k)^*$ corresponds to a block SDP problem. In addition, let
\begin{equation}\label{sec4-eq2}
(\textrm{TSSOS}) : \quad \theta_{tssos} := \sup_{\lambda}\{\lambda\mid f(\x)-\lambda\in\Sigma_{*}(\A)\}.
\end{equation}
Then we have the following hierarchy of lower bounds for the optimum of (P):
\begin{equation*}
\theta^*\ge\theta_{sos}=\theta_{tssos}\ge\cdots\ge\theta_2\ge\theta_1,
\end{equation*}
where the equality $\theta_{sos}=\theta_{tssos}$ follows from Theorem \ref{sec3-thm1}.

Let $\B$ be the monomial basis. For each $k\ge1$, the dual of $(\textrm{P}^k)^*$ is the following block moment problem
\begin{equation}\label{sec4-eq3}
(\textrm{P}^k):\quad
\begin{cases}
\inf\quad &L_{\y}(f)\\
\textrm{s.t.}\quad &B_{\A}^{(k)}\circ M_{\B}(\y)\succeq0,\\
&y_{\mathbf{0}}=1.
\end{cases}
\end{equation}

\newtssos{We call \eqref{sec4-eq1} and \eqref{sec4-eq3} the {\em TSSOS moment-SOS hierarchy} ({\em TSSOS hierarchy} in short) for the original problem (P) and call $k$ the {\em sparse order}.}

\begin{proposition}
For each $k\ge1$, there is no duality gap between ($\textrm{P}^k$) and $(\textrm{P}^k)^*$.
\end{proposition}
\begin{proof}
This easily follows from Proposition 3.1 of \cite{las1} for the dense case and the observation that each block of $B_{\A}^{(k)}\circ M_{\B}(\y)$ is a principal submatrix of $M_{\B}(\y)$.
\end{proof}

\begin{example}
Consider the polynomial $f=1+x_1^4+x_2^4+x_3^4+x_1x_2x_3+x_2$. A monomial basis for $f$ is $\{1,x_2,x_1^2,x_2^2,x_1x_3,x_3^2,x_1,x_2x_3,x_3,x_1x_2\}.$ Then
\begin{equation*}
C_{\A}^{(1)}=\begin{bmatrix}
1&1&1&1&0&1&0&0&0&0\\
1&1&0&0&1&0&0&0&0&0\\
1&0&1&1&0&1&0&0&0&0\\
1&0&1&1&0&1&0&0&0&0\\
0&1&0&0&1&0&0&0&0&0\\
1&0&1&1&0&1&0&0&0&0\\
0&0&0&0&0&0&1&1&0&0\\
0&0&0&0&0&0&1&1&0&0\\
0&0&0&0&0&0&0&0&1&1\\
0&0&0&0&0&0&0&0&1&1
\end{bmatrix}
\end{equation*}
and this yields
\begin{equation*}
B_{\A}^{(1)}=\begin{bmatrix}
1&1&1&1&1&1&0&0&0&0\\
1&1&1&1&1&1&0&0&0&0\\
1&1&1&1&1&1&0&0&0&0\\
1&1&1&1&1&1&0&0&0&0\\
1&1&1&1&1&1&0&0&0&0\\
1&1&1&1&1&1&0&0&0&0\\
0&0&0&0&0&0&1&1&0&0\\
0&0&0&0&0&0&1&1&0&0\\
0&0&0&0&0&0&0&0&1&1\\
0&0&0&0&0&0&0&0&1&1
\end{bmatrix}.
\end{equation*}
Furthermore, we have
\begin{equation*}
B_{\A}^{(2)}=C_{\A}^{(2)}=\begin{bmatrix}
1&1&1&1&1&1&0&0&0&0\\
1&1&1&1&1&1&0&0&0&0\\
1&1&1&1&1&1&0&0&0&0\\
1&1&1&1&1&1&0&0&0&0\\
1&1&1&1&1&1&0&0&0&0\\
1&1&1&1&1&1&0&0&0&0\\
0&0&0&0&0&0&1&1&1&1\\
0&0&0&0&0&0&1&1&1&1\\
0&0&0&0&0&0&1&1&1&1\\
0&0&0&0&0&0&1&1&1&1
\end{bmatrix}.
\end{equation*}
Thus $(B_{\A}^{(k)})_{k\ge1}$ stabilizes at $k=2$. Then solve the SDPs $(\textrm{P}^1)$, $(\textrm{P}^2)$ and we obtain $\theta_{1}=\theta_{2}=\theta_{tssos}=\theta_{sos}=\theta^*\approx0.4753$.
\end{example}

\medskip

\subsection*{Relationship with DSOS/SDSOS optimization}
The following definitions of DSOS and SDSOS have been introduced in \cite{ahmadi2014dsos}. For more details the interested reader is referred to \cite{ahmadi2014dsos}.

A symmetric matrix $Q\in\mathbb{S}^r$ is {\em diagonally dominant} if $Q_{ii}\ge\sum_{j\ne i}|Q_{ij}|$ for $i=1,\ldots,r$. We say that a polynomial $f(\x)\in\R[\x]$ is a {\em diagonally dominant sum of squares} (DSOS) polynomial if it admits a Gram matrix representation \eqref{sec2-eq2} with a diagonally dominant Gram matrix $Q$. We denote the set of DSOS polynomials by $DSOS$.

A symmetric matrix $Q\in\mathbb{S}^r$ is {\em scaled diagonally dominant} if there exists a positive definite $r\times r$ diagonal matrix $D$ such that $DAD$ is diagonally dominant. We say that a polynomial $f(\x)\in\R[\x]$ is a {\em scaled diagonally dominant sum of squares} (SDSOS) polynomial if it admits a Gram matrix representation \eqref{sec2-eq2} with a scaled diagonally dominant Gram matrix $Q$. We denote the set of SDSOS polynomials by $SDSOS$.

For a finite set $\A\subseteq\N^n$, let
$$\DSOS(\A):=\Sigma(\A)\cap DSOS$$
and
$$\SDSOS(\A):=\Sigma(\A)\cap SDSOS.$$
Clearly, it holds that $\DSOS(\A)\subseteq\SDSOS(\A)\subseteq\Sigma(\A)$.
\begin{theorem}\label{sec4-thm1}
For a finite set $\A\subseteq\N^n$, one has $\SDSOS(\A)\subseteq\Sigma_{1}(\A)$.
\end{theorem}
\begin{proof}
Let $\B$ be a monomial basis with $r=|\B|$. For any $f\in\SDSOS(\A)$, there exists a scaled diagonally dominant Gram matrix $Q\in\mathbb{S}_+^r$ indexed by $\B$. 
We then construct a matrix $\tilde{Q}\in\mathbb{S}^r$ by $\tilde{Q}=C_{\A}^{(1)}\circ Q$, i.e.,
\begin{equation*}
\tilde{Q}_{\b\g}=
\begin{cases}
Q_{\b\g},\quad&\textrm{if } \b+\g\in\A\cup2\B, \\
0,\quad&\textrm{otherwise}.
\end{cases}
\end{equation*}
\newtssos{By construction, $(\x^{\B})^TQ\x^{\B}-(\x^{\B})^T\tilde{Q}\x^{\B}=(\x^{\B})^T(A_{(\B+\B)\backslash(\A\cup2\B)}\circ Q)\x^{\B}=0$ since $(\x^{\B})^T(A_{(\B+\B)}\circ Q)\x^{\B}=f$ and $\supp(f)=\A$.
Thus $(\x^{\B})^T\tilde{Q}\x^{\B}=(\x^{\B})^TQ\x^{\B}=f$.} Note that we only replace off-diagonal entries by zeros in $Q$ and replacing off-diagonal entries by zeros does not affect the scaled diagonal dominance of a matrix. Hence $\tilde{Q}$ is also a scaled diagonally dominant matrix. Moreover, we have $B_{\A}^{(1)}\circ\tilde{Q}=C_{\A}^{(1)}\circ\tilde{Q}=\tilde{Q}$ by construction. Thus $f\in\Sigma_{1}(\A)$.
\end{proof}

Replacing the nonnegativity condition in (P') by the DSOS (resp. SDSOS) condition, we obtain the DSOS (resp. SDSOS) relaxation of (P):
\begin{equation*}
(\textrm{DSOS}) : \quad \theta_{dsos} := \sup_{\lambda}\{\lambda\mid f(\x)-\lambda\in\DSOS(\A)\}
\end{equation*}
and
\begin{equation*}
(\textrm{SDSOS}) : \quad \theta_{sdsos} := \sup_{\lambda}\{\lambda\mid f(\x)-\lambda\in\SDSOS(\A)\}.
\end{equation*}
The above DSOS and SDSOS relaxations for polynomial optimization have been introduced and studied in \cite{ahmadi2014dsos}.
By Theorem \ref{sec4-thm1}, we have the following hierarchy of lower bounds for the optimal value of (P):
\begin{equation*}
\theta^*\ge\theta_{sos}=\theta_{tssos}\ge\cdots\ge\theta_2\ge\theta_1\ge\theta_{sdsos}\ge\theta_{dsos}.
\end{equation*}

\section{A block Moment-SOS hierarchy for constrained POPs}\label{sec5}
In this section, we consider the constrained polynomial optimization problem:
\begin{equation*}
(\textrm{Q}) : \quad \theta^* := \inf_{\x}\{f(\x) : \x\in\mathbf{K}\}
\end{equation*}
where $f(\x)\in\R[\x]$ is a polynomial and $\mathbf{K}\subseteq\R^{n}$ is the basic semialgebraic set
\begin{equation}\label{sec5-eq2}
\mathbf{K} = \{\x\in\R^{n} : g_j(\x)\ge 0, j = 1,\ldots,m\},
\end{equation}
for some polynomials $g_j(\x)\in\R[\x], j = 1,\ldots,m.$

Let $d_j=\lceil\deg(g_j)/2\rceil,j=1,\ldots,m$ and  $d=\max\{\lceil\deg(f)/2\rceil,d_1,\ldots,d_m\}$ where $g_0:=1$. With $\hat{d}\ge d$ being a positive integer, the Lasserre hierarchy \cite{las1} of moment semidefinite relaxations of (Q) is defined by:
\begin{equation}\label{sec6-eq0}
(\textrm{Q}_{\hat{d}}):\quad
\begin{cases}
\inf\quad &L_{\y}(f)\\
\textrm{s.t.}\quad &M_{\hat{d}}(\y)\succeq0,\\
&M_{\hat{d}-d_j}(g_j\y)\succeq0,\quad j=1,\ldots,m,\\
&y_{\mathbf{0}}=1,
\end{cases}
\end{equation}
with optimal value denoted by $\theta_{\hat{d}}$ and we call $\hat{d}$ the {\em relaxation order}.
Let $\N^n_{2(\hat{d}-d_j)}$ be the standard monomial basis for $j=0,\ldots,m$. The dual of~\eqref{sec6-eq0} is an SDP equivalent to the following SOS problem:
\begin{equation}\label{sec6-sos}
(\textrm{Q}_{\hat{d}})^*:\quad
\begin{cases}
\sup\quad & \lambda\\
\textrm{s.t.}\quad & f - \lambda =  s_0+\sum_{j=1}^m s_jg_j \,,\\
& s_j \in \Sigma(\N_{2(\hat{d}-d_j)}^n), \quad j=0,\ldots,m.
\end{cases}
\end{equation}
Let
\begin{equation}\label{supp}
\A = \supp(f)\cup\bigcup_{j=1}^m\supp(g_j).
\end{equation}
Set $\S^{(0)}_{0,\hat{d}}=\A\cup(2\N)^n$ and $\S^{(0)}_{j,\hat{d}}=\emptyset,j=1,\ldots,m$.
Let us define $r_j:=\binom{n+\hat{d}-d_j}{\hat{d}-d_j}$.
For $k\ge1$, we recursively define binary matrices $B_{j,\hat{d}}^{(k)}\in\mathbb{S}^{r_j}\cap\Z_2^{r_j\times r_j}$, indexed by $\N^n_{\hat{d}-d_j}$, $j=0,\ldots,m$ via two successive steps:\\
\newtssos{1) {\bf Support-extension}: define a binary matrix $C_{j,\hat{d}}^{(k)}\in\mathbb{S}^{r_j}\cap\Z_2^{r_j\times r_j}$ with rows and columns indexed by $\N^n_{\hat{d}-d_j}$ by
\begin{equation}\label{sec6-eq7}
[C_{j,\hat{d}}^{(k)}]_{\b\g}:=
\begin{cases}
1,\quad&\textrm{if } (\supp(g_j)+\b+\g)\cap\bigcup_{j=0}^m\S_{j,\hat{d}}^{(k-1)}\ne\emptyset, \\
0,\quad&\textrm{otherwise}.
\end{cases}
\end{equation}
2) {\bf Block-closure}: let $B_{j,\hat{d}}^{(k)}=\overline{C_{j,\hat{d}}^{(k)}}$ and
\begin{equation}\label{sec6-eq9}
\S_{j,\hat{d}}^{(k)}:=\supp(g_j)+\bigcup_{[B_{j,\hat{d}}^{(k)}]_{\b\g}=1}\{\b+\g\}.
\end{equation}}
Therefore with $k\ge1$, we can further consider a block moment relaxations of ($\textrm{Q}_{\hat{d}}$) \eqref{sec6-eq0}:
\begin{equation}\label{sec-eq1}
(\textrm{Q}_{\hat{d}}^k):\quad
\begin{cases}
\inf\quad &L_{\y}(f)\\
\textrm{s.t.}\quad &B_{0,\hat{d}}^{(k)}\circ M_{\hat{d}}(\y)\succeq0,\\
&B_{j,\hat{d}}^{(k)}\circ M_{\hat{d}-d_j}(g_j\y)\succeq0,\quad j=1,\ldots,m,\\
&y_{\mathbf{0}}=1,
\end{cases}
\end{equation}
with optimal value denoted by $\theta^{(k)}_{\hat{d}}$. By construction, we have $\supp(B_{j,\hat{d}}^{(k)})\subseteq\supp(B_{j,\hat{d}}^{(k+1)})$ for all $k\ge1$ and $j=0,\ldots,m$. Hence the sequence of binary matrices $(B_{j,\hat{d}}^{(k)})_{k\ge1}$ stabilizes for all $j$ after a finite number of steps. We denote the stabilized matrices by $B_{j,\hat{d}}^{(*)},j=0,\ldots,m$ and denote the corresponding SDP problem \eqref{sec-eq1} by \newtssos{$(\textrm{Q}_{\hat{d}}^{\textrm{ts}})$} with optimal value $\theta^{*}_{\hat{d}}$.

\begin{theorem}\label{sec6-thm1}
For fixed $\hat{d}\ge d$, the sequence $(\theta^{(k)}_{\hat{d}})_{k\ge1}$ of optimal values of \eqref{sec-eq1} is monotone nondecreasing and in addition, $\theta^{*}_{\hat{d}}=\theta_{\hat{d}}$.
\end{theorem}
\begin{proof}
Since $\supp(B_{j,\hat{d}}^{(k)})\subseteq\supp(B_{j,\hat{d}}^{(k+1)})$ and $B_{j,\hat{d}}^{(k)}$ is block-diagonal (up to permutation) for all $j,k$, $(\textrm{Q}_{\hat{d}}^k)$ is a relaxation of $(\textrm{Q}_{\hat{d}}^{k+1})$ and $(\textrm{Q}_{\hat{d}})$. Therefore $(\theta^{(k)}_{\hat{d}})_{k\ge1}$ is nondecreasing and $\theta^{*}_{\hat{d}}\le\theta_{\hat{d}}$. 

Let $\S_{\hat{d}}^{(*)}=\cup_{j=0}^m(\supp(g_j)+\cup_{[B_{j,\hat{d}}^{(*)}]_{\b\g}=1}\{\b+\g\})$. Suppose that $\y=(y_{\a})_{\a\in\S_{\hat{d}}^{(*)}}$ is any feasible solution of
\newtssos{$(\textrm{Q}_{\hat{d}}^{\textrm{ts}})$}. Then define a sequence $\overline{\y}=(\overline{y}_{\a})_{\a\in\N^n_{2\hat{d}}}$ by
$$\overline{y}_{\a}=\begin{cases}y_{\a},\,\quad\textrm{if } \a\in\S_{\hat{d}}^{(*)},\\
0,\quad\quad\textrm{otherwise}.\end{cases}
$$
\newtssos{Because $B_{j,\hat{d}}^{(*)}$ is stabilized under the support-extension operation, by \eqref{sec6-eq7}, one has $(\supp(g_j)+\b+\g)\cap\S_{\hat{d}}^{(*)}=\emptyset$ for all $(\b,\g)\notin\supp(B_{j,\hat{d}}^{(*)})$ for $j=0,\ldots,m$.} Thus we have $M_{\hat{d}-d_j}(g_j\overline{\y})=B_{j,\hat{d}}^{(*)}\circ M_{\hat{d}-d_j}(g_j\y)$ for $j=0,\ldots,m$. Therefore $\overline{\y}$ is also a feasible solution of $(\textrm{Q}_{\hat{d}})$ and hence $L_{\y}(f)=L_{\overline{\y}}(f)\ge\theta_{\hat{d}}$.
Hence $\theta^{*}_{\hat{d}}\ge\theta_{\hat{d}}$ since $\y$ is an arbitrary feasible solution of \newtssos{$(\textrm{Q}_{\hat{d}}^{\textrm{ts}})$}. It follows $\theta^{*}_{\hat{d}}=\theta_{\hat{d}}$.
\end{proof}

\newtssos{\begin{theorem}\label{sec6-thm0}
For fixed $k\ge1$, the sequence $(\theta^{(k)}_{\hat{d}})_{\hat{d}\ge d}$ of optimal values of \eqref{sec-eq1} is monotone nondecreasing.
\end{theorem}
\begin{proof}
We only need to show that $\supp(B_{j,\hat{d}}^{(k)})\subseteq\supp(B_{j,\hat{d}+1}^{(k)})$ for all $j,k$ since this together with the fact that $B_{j,\hat{d}}^{(k)},B_{j,\hat{d}+1}^{(k)}$ are block-diagonal (up to permutation) implies that $(\textrm{Q}_{\hat{d}}^{k})$ is a relaxation of $(\textrm{Q}_{\hat{d}+1}^{k})$ and hence $\theta^{(k)}_{\hat{d}}\le\theta^{(k)}_{\hat{d}+1}$. Let us prove this conclusion by induction on $k$. For $k=1$, by $\eqref{sec6-eq7}$, we have $\supp(C_{j,\hat{d}}^{(1)})\subseteq\supp(C_{j,\hat{d}+1}^{(1)})$ for $j=0,\ldots,m$, which implies that $\supp(B_{j,\hat{d}}^{(1)})\subseteq\supp(B_{j,\hat{d}+1}^{(1)})$ for $j=0,\ldots,m$. Now assume that $\supp(B_{j,\hat{d}}^{(k)})\subseteq\supp(B_{j,\hat{d}+1}^{(k)})$, $j=0,\ldots,m$ hold for a given $k \geq 1$. By $\eqref{sec6-eq9}$ and by the induction hypothesis, we have $\S_{j,\hat{d}}^{(k)}\subseteq\S_{j,\hat{d}+1}^{(k)}$ for all $j$. Again by $\eqref{sec6-eq7}$, we have $\supp(C_{j,\hat{d}}^{(k+1)})\subseteq\supp(C_{j,\hat{d}+1}^{(k+1)})$ which implies $\supp(B_{j,\hat{d}}^{(k+1)})\subseteq\supp(B_{j,\hat{d}+1}^{(k+1)})$ for $j=0,\ldots,m$. Thus we complete the induction.
\end{proof}

Consequently combining Theorem \ref{sec6-thm1} and Theorem \ref{sec6-thm0}, we obtain the following two-level hierarchy of lower bounds for the optimal value of $(\textrm{Q})$:
\begin{equation}
\begin{matrix}
\theta^{(1)}_{d}&\le&\theta^{(2)}_{d}&\le&\cdots&\le&\theta^{*}_{d}=\theta_{d}\\
\vge&&\vge&&&&\vge\\
\theta^{(1)}_{d+1}&\le&\theta^{(2)}_{d+1}&\le&\cdots&\le&\theta^{*}_{d+1}=\theta_{d+1}\\
\vge&&\vge&&&&\vge\\
\vdots&&\vdots&&\vdots&&\vdots\\
\vge&&\vge&&&&\vge\\
\theta^{(1)}_{\hat{d}}&\le&\theta^{(2)}_{\hat{d}}&\le&\cdots&\le&\theta^{*}_{\hat{d}}=\theta_{\hat{d}}\\
\vge&&\vge&&&&\vge\\
\vdots&&\vdots&&\vdots&&\vdots\\
\end{matrix}
\end{equation}}

For each $j=1,\ldots,m$, writing $M_{\hat{d}-d_j}(g_j\y)=\sum_{\a}D_{\a}^jy_{\a}$ for appropriate symmetric matrices $\{D_{\a}^j\}$, then the dual of $(\textrm{Q}_{\hat{d}}^k)$ reads as
\begin{equation}\label{sec6-eq1}
(\textrm{Q}_{\hat{d}}^k)^*:\quad
\begin{cases}
\sup\quad &\lambda\\
\textrm{s.t.}\quad &\langle Q_0,A_{\a}\rangle+\sum_{j=1}^m\langle  Q_j,D_{\a}^j\rangle+\lambda\delta_{\mathbf{0}\a}=f_{\a},\forall\a\in\S_{\hat{d}}^{(k)},\\
&Q_j\in\mathbb{S}_+^{r_j}(B_{j,\hat{d}}^{(k)}),\quad j=0,\ldots,m,
\end{cases}
\end{equation}
where $\S_{\hat{d}}^{(k)}=\cup_{j=0}^m(\supp(g_j)+\cup_{[B_{j,\hat{d}}^{(k)}]_{\b\g}=1}\{\b+\g\})$, $A_{\a}$ is defined in Section~\ref{sec3} and $\delta_{\mathbf{0}\a}$ is the usual Kronecker symbol.

\newtssos{We call \eqref{sec-eq1} and \eqref{sec6-eq1} the {\em TSSOS moment-SOS hierarchy} ({\em TSSOS hierarchy} in short) for the original problem (Q) and call $k$ the {\em sparse order}.}

\begin{proposition}
Let $f\in\R[\x]$ and $\mathbf{K}$ be as in \eqref{sec5-eq2}. Assume that $K$ has a nonempty interior. Then there is no duality gap between $(\textrm{Q}_{\hat{d}}^k)$ and $(\textrm{Q}_{\hat{d}}^k)^*$ for any $\hat{d}\ge d$ and $k\ge1$.
\end{proposition}
\begin{proof}
By the duality theory of convex programming, this easily follows from Theorem 4.2 of \cite{las1} for the dense case and the observation that each block of $B_{j,\hat{d}}^{(k)}\circ M_{\hat{d}-d_j}(g_j\y)$ is a principal submatrix of $M_{\hat{d}-d_j}(g_j\y)$ for all $j,k$.
\end{proof}

For any feasible solution of $(\textrm{Q}_{\hat{d}}^k)^*$, multiplying each side of the constraint in \eqref{sec6-eq1} by $\x^{\a}$ for all $\a\in\N^n_{2\hat{d}}$ and summing up yields
\begin{equation}\label{sec6-eq2}
\langle Q_0,\sum_{\a\in\N^n_{2\hat{d}}}A_{\a}\x^{\a}\rangle+\sum_{j=1}^m\langle Q_j,\sum_{\a\in\N^n_{2\hat{d}}}D_{\a}^j\x^{\a}\rangle=f-\lambda.
\end{equation}
Note that $\sum_{\a\in\N^n_{2\hat{d}}}A_{\a}\x^{\a}=\x^{\N^n_{\hat{d}}}\cdot(\x^{\N^n_{\hat{d}}})^T$ and $\sum_{\a\in\N^n_{2\hat{d}}}D_{\a}^j\x^{\a}=g_j\x^{\N^n_{\hat{d}-d_j}}\cdot(\x^{\N^n_{\hat{d}-d_j}})^T$ for $j=1\ldots,m$. Hence we can rewrite (\ref{sec6-eq2}) as
\begin{equation}\label{sec6-eq3}
(\x^{\N^n_{\hat{d}}})^TQ_0\x^{\N^n_{\hat{d}}}+\sum_{j=1}^mg_j(\x^{\N^n_{\hat{d}-d_j}})^TQ_j\x^{\N^n_{\hat{d}-d_j}}=f-\lambda.
\end{equation}
For each $j$, the binary matrix $B_{j,\hat{d}}^{(k)}$ is block-diagonal up to permutation and $B_{j,\hat{d}}^{(k)}$ induces a partition of the monomial basis  $\N^n_{\hat{d}-d_j}$: two vectors $\b,\g \in \N^n_{\hat{d}-d_j}$ belong to the same block if and only if the rows and columns indexed by $\b,\g$ belong to the same block in $B_{j,\hat{d}}^{(k)}$. If some diagonal element of $B_{j,\hat{d}}^{(k)}$ is zero, then the corresponding basis element can be discarded. Assume that $v_{j1}(\x),\ldots,v_{jl_j}(\x)$ are the resulting blocks in this partition and $Q_{j1},\ldots,Q_{jl_j}$ are the corresponding principal submatrices of $Q_j$. Then \eqref{sec6-eq3} reads as
\begin{equation}\label{sec6-eq4}
\sum_{i=1}^{l_0}v_{ji}(\x)^TQ_{ji}v_{ji}(\x)+\sum_{j=1}^mg_j\sum_{i=1}^{l_j}v_{ji}(\x)^TQ_{ji}v_{ji}(\x)=f-\lambda.
\end{equation}
For all $i,j$, the polynomial $s_{ji}:=v_{ji}(\x)^TQ_{ji}v_{ji}(\x)$ is an SOS polynomial since $Q_{ji}$ is PSD. Then we have
\begin{equation}\label{sec6-eq5}
\sum_{i=1}^{l_0}s_{ji}+\sum_{j=1}^mg_j\sum_{i=1}^{l_j}s_{ji}=f-\lambda.
\end{equation}
Notice that \eqref{sec6-eq5} is in fact a {\em sparse} Putinar's representation for the polynomial $f-\lambda$.
This representation is a certificate of positivity on $\mathbf{K}$ for the polynomial $f-\lambda$. Indeed \eqref{sec6-eq5} ensures that $f-\lambda$ is nonnegative on $\mathbf{K}$ and each SOS $s_{ji}$ has an associated Gram matrix  $Q_{ji}$ indexed in the sparse monomial basis $v_{ji}(\x)$.

\begin{example}
Let $f=x_1^4+x_2^4-x_1x_2$ and $\mathbf{K} = \{(x_1,x_2)\in\R^{2} : g_1 = 1-2x_1^2-x_2^2\ge 0\}$. Let $\A=\{(4,0),(0,4),(1,1),(0,0),(2,0),(0,2)\}$ and $\hat{d}=2$. Take $\{1,x_1,x_2,x_1^2,x_1x_2,x_2^2\}$ as a monomial basis. Then
\begin{equation*}
C_{0,2}^{(1)}=\begin{bmatrix}
1&0&0&1&1&1\\
0&1&1&0&0&0\\
0&1&1&0&0&0\\
1&0&0&1&0&1\\
1&0&0&0&1&0\\
1&0&0&1&0&1
\end{bmatrix}
\quad\textrm{ and }\quad
C_{1,2}^{(1)}=\begin{bmatrix}
1&0&0\\
0&1&1\\
0&1&1
\end{bmatrix}.
\end{equation*}
This yields
\begin{equation*}
B_{0,2}^{(1)}=\begin{bmatrix}
1&0&0&1&1&1\\
0&1&1&0&0&0\\
0&1&1&0&0&0\\
1&0&0&1&1&1\\
1&0&0&1&1&1\\
1&0&0&1&1&1
\end{bmatrix}
\quad\textrm{ and }\quad
B_{1,2}^{(1)}=\begin{bmatrix}
1&0&0\\
0&1&1\\
0&1&1
\end{bmatrix}.
\end{equation*}
Furthermore, we have $B_{j,2}^{(2)}=C_{j,2}^{(1)}=B_{j,2}^{(1)},j=1,2$.
Thus $(B_{0,2}^{(k)},B_{1,2}^{(k)})_{k\ge1}$ stabilizes at $k=1$ and ($\textrm{Q}_2^1$) can be read as
\begin{equation*}
(\textrm{Q}_2^1):\quad
\begin{cases}
\inf\quad &y_{40}+y_{04}-y_{11}\\
\textrm{s.t.}\quad &\begin{bmatrix}
y_{00}&&&y_{20}&y_{11}&y_{02}\\
&y_{20}&y_{11}&&&\\
&y_{11}&y_{02}&&&\\
y_{20}&&&y_{40}&y_{31}&y_{22}\\
y_{11}&&&y_{31}&y_{22}&y_{13}\\
y_{02}&&&y_{22}&y_{13}&y_{04}
\end{bmatrix}\succeq0,\\
&\begin{bmatrix}
y_{00}-2y_{20}-y_{02}&&\\
&y_{20}-2y_{40}-y_{22}&y_{11}-2y_{31}-y_{13}\\
&y_{11}-2y_{31}-y_{13}&y_{02}-2y_{22}-y_{04}
\end{bmatrix}\succeq0,\\
&y_{00}=1.
\end{cases}
\end{equation*}
We have $\theta_2^{(1)}=\theta_2^{*}=\theta_2=\theta^*=-0.125$.
\end{example}


\section{Sign-symmetries and a sparse representation theorem for positive polynomials}\label{sec6}
Suppose that the binary matrix $B_{0,\hat{d}}^{(*)}$ is not an all-one matrix. Then as was already noted in Section~\ref{sec5}, the block-diagonal (up to permutation) matrix $B_{0,\hat{d}}^{(*)}$ induces a partition of the monomial basis $\N^n_{\hat{d}}$: two vectors $\b,\g\in\N^n_{\hat{d}}$ belong to the same block if and only if the rows and columns indexed by $\b,\g$ belong to the same block in $B_{0,\hat{d}}^{(*)}$. We next provide an interpretation of this partition in terms of {\em sign-symmetries}, a tool introduced in \cite{lo1} to characterize block-diagonal SOS decompositions for nonnegative polynomials.
\begin{definition}
Given a finite set $\A\subseteq\N^n$, the {\em sign-symmetries} of $\A$ are defined by all vectors $\br\in\Z_2^n$ such that $\br^T\a\equiv0$ $(\textrm{mod }2)$ for all $\a\in\A$.
\end{definition}
\newtssos{For any $\a\in\N^n$, we define $(\a)_2:=(\alpha_1 (\textrm{mod } 2), \dots, \alpha_n (\textrm{mod } 2))\in\Z_2^n$.
We also use the same notation for any subset $\A\subseteq\N^n$, i.e.\,$(\A)_2:=\{(\a)_2\mid\a\in\mathscr{A}\}\subseteq\Z_2^n$.

For a subset $S\subseteq\Z_2^n$, the {\em subspace} spanned by $S$ in $\Z_2^n$, denoted by $\overline{S}$, is the set $\{(\sum_i\bs_i)_2\mid\bs_i\in S\}$ and the {\em orthogonal complement space} of $S$ in $\Z_2^n$, denoted by $S^{\perp}$, is the set $\{\a\in\Z_2^n\mid\a^T\bs\equiv0\,(\textrm{mod }2)\,,\forall\bs\in S\}$.
\begin{remark}
By definition, the set of sign-symmetries of $\A$ is just the orthogonal complement space $(\A)_2^{\perp}$ in $\Z_2^n$. Hence the sign-symmetries of $\A$ can be essentially represented by a basis of the subspace $(\A)_2^{\perp}$ in $\Z_2^n$.
\end{remark}

\begin{lemma}\label{sec4-lm}
Let $S\subseteq\Z_2^n$. Then $(S^{\perp})^{\perp}=\overline{S}$.
\end{lemma}
\begin{proof}
It is immediate from the definitions.
\end{proof}

For an (undirected) graph $G(V,E)$ with $V\subseteq\N^n$, define $\supp(G):=\{\b+\g\mid\{\b,\g\}\in E\}$ and we also use $E(G)$ to indicate the edge set of $G$.
\begin{lemma}\label{sec4-lm2}
Suppose $B$ is a $\{0,1\}$-binary matrix with rows and columns indexed by $\B\subseteq\N^n$ and $G$ is its adjacency graph. Let $\overline{G}$ be the adjacency graph of $\overline{B}$. Then $(\supp(\overline{G}))_2\subseteq\overline{(\supp(G))_2}$.
\end{lemma}
\begin{proof}
By definition, for any $\{\b,\g\}\in E(\overline{G})$, we need to show $(\b+\g)_2\in\overline{(\supp(G))_2}$. By virtue of the graphical description of block-closure, $\b,\g$ belong to the same connected component of $G$. Therefore there is a path connecting $\b$ and $\g$ in $G$: $\{\b,\up_1,\ldots,\up_r,\g\}$ with $\{\b,\up_1\},\{\up_r,\g\}\in E(G)$ and $\{\up_i,\up_{i+1}\}\in E(G),i=1,\ldots,r-1$. From $(\b+\up_1)_2,(\up_1+\up_2)_2\in (\supp(G))_2$, we deduce that $(\b+\up_2)_2\in\overline{(\supp(G))_2}$. Likewise, we can prove $(\b+\up_i)_2\in\overline{(\supp(G))_2}$ for $i=3,\ldots,r+1$ with $\up_{r+1}:=\g$. Hence $(\b+\g)_2\in\overline{(\supp(G))_2}$ as desired.
\end{proof}}

\begin{theorem}\label{sec6-thm2}
For a positive integer $\hat{d}$ and a finite set $\A\subseteq\N^n_{2\hat{d}}$, let $\N^n_{\hat{d}}$ be the standard monomial basis and let us define the sign-symmetries of $\A$ with the columns of the binary matrix $R$. Let $B_{0,\hat{d}}^{(*)}$ be defined as in Section~\ref{sec5}. Then $\b,\g$ belong to the same block in the partition of $\N^n_{\hat{d}}$ induced by $B_{0,\hat{d}}^{(*)}$ if and only if $R^T(\b+\g)\equiv0$ $(\textrm{mod }2)$.
\end{theorem}
\begin{proof}
Let $G(V,E)$ be the adjacency graph of $B_{0,\hat{d}}^{(*)}$ with vertices $V=\N^n_{\hat{d}}$ and edges $E=\{\{\b,\g\}\mid[B_{0,\hat{d}}^{(*)}]_{\b\g}=1\}$. Then the partition of $\N^n_{\hat{d}}$ induced by $B_{0,\hat{d}}^{(*)}$ corresponds to the connected components of $G$.
Note that every connected component of $G$ is a complete subgraph.

{\bf Claim \uppercase\expandafter{\romannumeral1}.} If $\a\in\supp(G)$, then for any $\a'\in\N^n_{2\hat{d}}$ with $(\a')_2=(\a)_2$, one has $\a'\in\supp(G)$.

{\em Proof of Claim \uppercase\expandafter{\romannumeral1}}. Suppose $\a\in\supp(G)$. If $(\a)_2=(\a')_2=\mathbf{0}$, assume $\a'=\b+\g$ for some $\b,\g\in\N^n_{\hat{d}}$. Then $\b+\g\in(2\N)^n$. Hence $\{\b,\g\}\in E(G)$ and it follows $\a'\in\supp(G)$. Now assume $(\a)_2\ne\mathbf{0}$. For $\bs=(s_i),\bs'=(s_i')\in\Z_2^n$, let $\tau(\bs):=\sum_{i=1}^n s_i$ and we use $\bs\perp\bs'$ to indicate that $s_i=s_i'=1$ holds for no $i$. If $\tau((\a)_2)$ is odd, let $\bs_1,\bs_2\in\Z_2^n\cap\N^n_{\hat{d}}$ such that $(\a)_2=\bs_1+\bs_2$ and $\bs_1\perp\bs_2$. If $\tau((\a)_2)$ is even, we further require that $\tau(\bs_1),\tau(\bs_2)$ have the same parity as $\hat{d}$. It is easy to check that such $\bs_1,\bs_2$ always exist. Then there must exist $\b_1,\b_2\in(2\N)^n$ such that $\bs_1+\b_1,\bs_2+\b_2\in\N^n_{\hat{d}}$ and $\a=(\bs_1+\b_1)+(\bs_2+\b_2)$. It follows that $\{\bs_1+\b_1,\bs_2+\b_2\}\in E(G)$ since $\a\in\supp(G)$ and $B_{0,\hat{d}}^{(*)}$ is stabilized under the support-extension operation. Because $(\a')_2=(\a)_2$, there must exist $\b_1',\b_2'\in(2\N)^n$ such that $\bs_1+\b_1',\bs_2+\b_2'\in\N^n_{\hat{d}}$ and $\a'=(\bs_1+\b_1')+(\bs_2+\b_2')$. Note that $(\bs_1+\b_1)+(\bs_1+\b_1')\in(2\N)^n$ and $(\bs_2+\b_2)+(\bs_2+\b_2')\in(2\N)^n$. Hence $\{\bs_1+\b_1,\bs_1+\b_1'\},\{\bs_2+\b_2,\bs_2+\b_2'\}\in E(G)$, which together with $\{\bs_1+\b_1,\bs_2+\b_2\}\in E(G)$ implies that $\bs_1+\b_1$, $\bs_1+\b_1'$, $\bs_2+\b_2$, $\bs_2+\b_2'$ belong to the same connected component of $G$. So $\{\bs_1+\b_1',\bs_2+\b_2'\}\in E(G)$ and $\a'\in\supp(G)$. The proof of Claim \uppercase\expandafter{\romannumeral1} is finished.

{\bf Claim \uppercase\expandafter{\romannumeral2}.} Let $S=(\A)_2$. The edge set of $G$ is
\begin{equation*}
    E=\{\{\b,\g\}\in V^2\mid(\b+\g)_2\in\overline{S}\},
\end{equation*}
which is equivalent to (by Claim \uppercase\expandafter{\romannumeral1} and the fact that $B_{0,\hat{d}}^{(*)}$ is stabilized under the support-extension operation)
\begin{equation*}
    \supp(G)=\{\a\in\N^n_{2\hat{d}}\mid(\a)_2\in\overline{S}\}.
\end{equation*}

{\em Proof of Claim \uppercase\expandafter{\romannumeral2}}. First we prove that $\supp(G)\subseteq\{\a\in\N^n_{2\hat{d}}\mid(\a)_2\in\overline{S}\}$.
For $j=0,\ldots,m$, let $\S_{j,\hat{d}}^{(k)},C_{j,\hat{d}}^{(k)},B_{j,\hat{d}}^{(k)}$ be defined as in Section\,\ref{sec5} and let $H_j^k,G_j^k$ be the adjacency graphs of $C_{j,\hat{d}}^{(k)},B_{j,\hat{d}}^{(k)}$, respectively.
By construction, one has $\supp(G)=\bigcup_{k\ge0}\bigcup_{j=0}^m\S_{j,\hat{d}}^{(k)}$. It suffices to prove
\begin{equation}\label{sec7-eq2}
\bigcup_{j=0}^m\S_{j,\hat{d}}^{(k)}\subseteq\{\a\in\N^n_{2\hat{d}}\mid(\a)_2\in\overline{S}\}
\end{equation}
for all $k$. Let us do induction on $k\ge0$. It is obvious that \eqref{sec7-eq2} is valid for $k=0$. Now assume that \eqref{sec7-eq2} holds for a given $k \geq 0$. 
For $0\le j\le m$ and for any $\a'\in\supp(H_j^{k+1})$, by \eqref{sec6-eq7} we have $(\supp(g_j)+\a')\cap\bigcup_{j=0}^m\S_{j,\hat{d}}^{(k)}\ne\emptyset$, which implies that $(\supp(g_j)+\a')\cap\{\a\in\N^n_{2\hat{d}}\mid(\a)_2\in\overline{S}\}\ne\emptyset$ by the induction hypothesis.
It follows that $(\a')_2\in\overline{S}$. Thus $\supp(H_j^{k+1})\subseteq\{\a\in\N^n_{2\hat{d}}\mid(\a)_2\in\overline{S}\}$. 
Then by Lemma \ref{sec4-lm2},  $\supp(G_j^{k+1})\subseteq\{\a\in\N^n_{2\hat{d}}\mid(\a)_2\in\overline{S}\}$.
By \eqref{sec6-eq9}, $\S_{j,\hat{d}}^{(k+1)}=\supp(g_j)+\supp(G_j^{k+1})$. Hence $\S_{j,\hat{d}}^{(k+1)}\subseteq\{\a\in\N^n_{2\hat{d}}\mid(\a)_2\in\overline{S}\}$ for all $j$. This completes the induction.

Next we need to prove that $\{\a\in\N^n_{2\hat{d}}\mid(\a)_2\in\overline{S}\}\subseteq\supp(G)$, or equivalently
\begin{equation}\label{sec7-eq3}
\overline{S}\cap\N^n_{2\hat{d}}\subseteq(\supp(G))_2.
\end{equation}
For any $\bs\in\overline{S}\cap\N^n_{2\hat{d}}$, we can write $\bs=(\sum_{i=1}^l\bs_i)_2$ for some $\{\bs_i\}_i\subseteq S$. Let us prove \eqref{sec7-eq3} by induction on $l$. The case of $l=1$ follows from $\bs_1\in S\subseteq(\supp(G))_2$. Now assume that $(\sum_{i=1}^l\bs_i)_2\in(\supp(G))_2$. Suppose $(\sum_{i=1}^l\bs_i)_2=(p_s)_{s=1}^n$ and $\bs_{l+1}=(q_s)_{s=1}^n$. \newtssos{Let $J_1=\{s\mid p_s=1,q_s=0\}$, $J_2=\{s\mid p_s=q_s=1\}$ and $J_3=\{s\mid p_s=0,q_s=1\}$. If $|J_1|,|J_2|,|J_3|\le\hat{d}$, let $I=J_2$; if $|J_2|>\hat{d}$, let $I$ be any $\hat{d}$-subset of $J_2$; if $|J_1|>\hat{d}$, let $K$ be any $\hat{d}$-subset of $J_1$ and $I=J_1\backslash K$; if $|J_3|>\hat{d}$, let $K$ be any $\hat{d}$-subset of $J_3$ and $I=J_3\backslash K$. Then define $\bu=(u_s)\in\Z_2^n\cap\N^n_{\hat{d}}$ by
\begin{equation*}
u_s=
\begin{cases}
1,\quad&s\in I, \\
0,\quad&\textrm{otherwise},
\end{cases}
\end{equation*}
and let $\bv=(\sum_{i=1}^l\bs_i+\bu)_2$, $\bo=(\bs_{l+1}+\bu)_2$.
Then $(\sum_{i=1}^l\bs_i)_2=(\bu+\bv)_2$ and $(\bs_{l+1})_2=(\bu+\bo)_2$.
In the case of $|J_1|,|J_2|,|J_3|\le\hat{d}$, one has $\tau(\bv)=|J_1|\le\hat{d}$ and $\tau(\bo)=|J_3|\le\hat{d}$; in the case of $|J_2|>\hat{d}$, one has $\tau(\bv)=|J_1|+|J_2|-\hat{d}\le\hat{d}$ and $\tau(\bo)=|J_3|+|J_2|-\hat{d}\le\hat{d}$ because $(\sum_{i=1}^l\bs_i)_2,\bs_{l+1}\in\N^n_{2\hat{d}}$; in the case of $|J_1|>\hat{d}$, one has $\tau(\bv)=|J_1|+|J_2|-\hat{d}\le\hat{d}$ and $\tau(\bo)=|J_3|+|J_1|-\hat{d}\le\hat{d}$ because $(\sum_{i=1}^l\bs_i)_2,\bs=(\sum_{i=1}^{l+1}\bs_i)_2\in\N^n_{2\hat{d}}$; in the case of $|J_3|>\hat{d}$, one has $\tau(\bv)=|J_1|+|J_3|-\hat{d}\le\hat{d}$ and $\tau(\bo)=|J_3|+|J_2|-\hat{d}\le\hat{d}$ because $\bs=(\sum_{i=1}^{l+1}\bs_i)_2,\bs_{l+1}\in\N^n_{2\hat{d}}$. Consequently, $\bv,\bo\in\N^n_{\hat{d}}$. 
}
By the induction hypothesis, $(\bu+\bv)_2\in(\supp(G))_2$ which implies $\bu+\bv\in\supp(G)$ by Claim \uppercase\expandafter{\romannumeral1} and hence $\{\bu,\bv\}\in E$ (because $B_{0,\hat{d}}^{(*)}$ is stabilized under the support-extension operation). We also have $(\bu+\bo)_2\in S\subseteq(\supp(G))_2$ which implies $\bu+\bo\in\supp(G)$ by Claim \uppercase\expandafter{\romannumeral1} and hence $\{\bu,\bo\}\in E$. It follows that $\{\bv,\bo\}\in E$ and $\bv+\bo\in\supp(G)$. Thus $(\sum_{i=1}^{l+1}\bs_i)_2=(\bv+\bo)_2\in(\supp(G))_2$ which completes the induction and also completes the proof of Claim \uppercase\expandafter{\romannumeral2}.

By Lemma \ref{sec4-lm}, we have $\overline{S}=R^{\perp}$. Thus $\b,\g$ belong to the same connected component of $G$ if and only if $(\b+\g)_2\in\overline{S}$ by Claim \uppercase\expandafter{\romannumeral2} which is equivalent to $R^T(\b+\g)\equiv0$ $(\textrm{mod }2)$.
\end{proof}

\begin{remark}
Note that Theorem \ref{sec6-thm2} is applied for the standard monomial basis $\N^n_{\hat{d}}$. If a smaller monomial basis is chosen, then we only have the ``only if" part of the conclusion in Theorem \ref{sec6-thm2}. See Example \ref{sec6-ex1}.
\end{remark}

\newtssos{\begin{example}\label{sec6-ex1}
Let $f=1+x^2y^4+x^4y^2+x^4y^4-xy^2-3x^2y^2$ and $\A=\supp(f)$. The monomial basis given by the Newton polytope method is $\B=\{1,xy,xy^2,x^2y,x^2y^2\}$. The sign-symmetries of $\A$ consist of two elements: $(0,0)$ and $(1,0)$. According to the sign-symmetries, $\B$ is partitioned into $\{1,xy^2,x^2y^2\}$ and $\{xy,x^2y\}$. On the other hand, we have
\begin{equation*}
C_{\A}^{(1)}=\begin{bmatrix}
1&0&1&0&1\\
0&1&0&0&0\\
1&0&1&0&0\\
0&0&0&1&0\\
1&0&0&0&1
\end{bmatrix}
\quad\textrm{ and }\quad
B_{\A}^{(*)}=B_{\A}^{(1)}=\begin{bmatrix}
1&0&1&0&1\\
0&1&0&0&0\\
1&0&1&0&1\\
0&0&0&1&0\\
1&0&1&0&1
\end{bmatrix}.
\end{equation*}
Thus the partition of $\B$ induced by $B_{\A}^{(*)}$ is $\{1,xy^2,x^2y^2\}$, $\{xy\}$ and $\{x^2y\}$, which is a refinement of the partition determined by the sign-symmetries.
\end{example}

By virtue of Theorem \ref{sec6-thm2}, the partition of the monomial basis $\N^n_{\hat{d}-d_j}$ induced by $B_{j,\hat{d}}^{(*)}$, $j=1\ldots,m$, can also be characterized using sign-symmetries. 
\begin{corollary}\label{sec6-col}
Notations are as in Theorem \ref{sec6-thm2}. Let $B_{j,\hat{d}}^{(*)}$ be defined as in Section~\ref{sec5}. Then $\b,\g$ belong to the same block in the partition of $\N^n_{\hat{d}-d_j}$ induced by $B_{j,\hat{d}}^{(*)}$ if and only if $R^T(\b+\g)\equiv0$ $(\textrm{mod }2)$, $j=1\ldots,m$.
\end{corollary}
\begin{proof}
Let $G_j(V_j,E_j)$ be the adjacency graph of $B_{j,\hat{d}}^{(*)}$ with vertices $V_j=\N^n_{\hat{d}-d_j}$ and edges $E_j=\{\{\b,\g\}\mid[B_{j,\hat{d}}^{(*)}]_{\b\g}=1\}$, $j=1,\ldots,m$. Then the partition of $\N^n_{\hat{d}-d_j}$ induced by $B_{j,\hat{d}}^{(*)}$ corresponds to the connected components of $G_j$. Note also that every connected component of $G_j$ is a complete subgraph.

If $\b,\g$ belong to the same connected component of $G_j$, then $\{\b,\g\}\in E_j$. So $\b+\g+\supp(g_j)\subseteq\supp(G)$ which implies $(\b+\g)_2\in(\supp(G))_2$ and hence $(\b+\g)_2\in\overline{S}$ by Claim \uppercase\expandafter{\romannumeral2} in the proof of Theorem \ref{sec6-thm2}. Thus $R^T(\b+\g)\equiv0$ $(\textrm{mod }2)$.

If $\b,\g$ don't belong to the same connected component of $G_j$, then $\{\b,\g\}\notin E_j$. So $\b+\g+\supp(g_j)\not\subseteq\supp(G)$ which implies  $(\b+\g)_2\notin\overline{S}$ by Claim \uppercase\expandafter{\romannumeral2} in the proof of Theorem \ref{sec6-thm2}. Thus $R^T(\b+\g)\not\equiv0$ $(\textrm{mod }2)$.
\end{proof}

Theorem \ref{sec6-thm2} together with Corollary \ref{sec6-col} implies that the block-structure of the TSSOS hierarchy at each relaxation order (if the standard monomial bases $\N^n_{\hat{d}-d_j},j=0,\ldots,m$ are used) converges to the block-structure determined by the sign-symmetries related to the support of the input data.

\begin{remark}
Though it is guaranteed that at the final iterative step of the TSSOS hierarchy, an equivalent SDP (with block-structure determined by sign-symmetries if the standard monomial bases are used) is retrieved, in practice it frequently happens that the same optimal value as the dense moment-SOS relaxation is achieved at an earlier step, even at the first step, but with a much cheaper computational cost as we can see in Section~\ref{sec:benchs}.
\end{remark}}


For a family of polynomials $\mathbf{g}=(g_1,\ldots,g_m)\subseteq\R[\x]$, the associated {\em quadratic module} $\mathcal{Q}(\mathbf{g})=\mathcal{Q}(g_1,\ldots,g_m)\subseteq\R[\x]$ is defined by
\begin{equation}
    \mathcal{Q}(\mathbf{g}):=\{s_0+\sum_{j=1}^ms_jg_j\mid s_j\textrm{ is an SOS}, \, j=0,\ldots,m\}.
\end{equation}
The quadratic module $\mathcal{Q}(\mathbf{g})$ associated with $\mathbf{K}$ in \eqref{sec5-eq2} is said to be {\em Archimedean} if there exists $N>0$ such that the quadratic polynomial $\x\mapsto N-\|\x\|^2$ belongs to $\mathcal{Q}(\mathbf{g})$.

As a corollary of Theorem \ref{sec6-thm2} and Corollary \ref{sec6-col}, we
obtain the following sparse representation theorem for positive polynomials over basic compact semialgebraic sets.

\begin{theorem}\label{sec6-thm3}
Let $f\in\R[\x]$ and $\mathbf{K}$ be as in \eqref{sec5-eq2}. Assume that the quadratic module $\mathcal{Q}(\mathbf{g})$ is Archimedean and that $f$ is positive on $\mathbf{K}$. Let $\A=\supp(f)\cup\bigcup_{j=1}^m\supp(g_j)$ and let us define the sign-symmetries of $\A$ with the columns of the binary matrix $R$. Then $f$ can be represented as
\begin{equation*}
f=s_0+\sum_{j=1}^ms_jg_j,
\end{equation*}
for some SOS polynomials $s_0,s_1,\ldots,s_m$ satisfying $R^T\a\equiv0$ $(\textrm{mod }2)$ for any $\a\in\supp(s_j),j=0,\ldots,m$.
\end{theorem}
\begin{proof}
By Putinar's Positivstellensatz \cite{pu}, there exist SOS polynomials $t_0,t_1,\ldots,t_m$ such that
\begin{equation}\label{sec7-eq1}
f=t_0+\sum_{j=1}^mt_jg_j.
\end{equation}
Let $d_j=\lceil\deg(g_j)/2\rceil,j=0,\ldots,m$ and $\hat{d}=\max\{\lceil\deg(t_jg_j)/2\rceil:j=0,1,\ldots,m\}$ with $g_0=1$. Let $Q_j$ be a Gram matrix associated to $t_j$ and indexed by the monomial basis $\N^n_{\hat{d}-d_j},j=0,\ldots,m$. Then set $s_j=(\x^{\N^n_{\hat{d}-d_j}})^T(B_{j,\hat{d}}^{(*)}\circ Q_j)\x^{\N^n_{\hat{d}-d_j}}$ for $j=0,\ldots,m$, where $B_{j,\hat{d}}^{(*)}$ is defined as in Section~\ref{sec5}. For all $j=0,\ldots,m$, $B_{j,\hat{d}}^{(*)}\circ Q_j$ is block-diagonal up to permutation and $Q_j$ is positive semidefinite, thus $s_j$ is an SOS polynomial.

Following the notation from Theorem \ref{sec6-thm2} and Corollary \ref{sec6-col}, let $G$ be the adjacency graph of $B_{0,\hat{d}}^{(*)}$. By construction, $\supp(s_0)\subseteq\supp(G)$.
For $j=1,\ldots,m$, let $B_{j,\hat{d}}^{(k)},B_{j,\hat{d}}^{(*)}$ be defined as in Section\,\ref{sec5} and let $G_j^k,G_j$ be the adjacency graphs of $B_{j,\hat{d}}^{(k)},B_{j,\hat{d}}^{(*)}$, respectively. By construction, $\supp(G_j)=\bigcup_{k\ge1}\supp(G_j^k)$. By the proof of Claim \uppercase\expandafter{\romannumeral2} in Theorem \ref{sec6-thm2}, $\supp(G_j^k)\subseteq\supp(G)$ for all $k\ge1$.
It follows that $\supp(G_j)\subseteq\supp(G)$ for $j=1,\ldots,m$.
Therefore, we have $\supp(s_j)\subseteq\supp(G_j)\subseteq\supp(G)$ for $1\le j\le m$. Hence for any $j$ and any $\a\in\supp(s_j)$, one has $(\a)_2\in\overline{S}$ by Claim \uppercase\expandafter{\romannumeral2} in the proof of Theorem \ref{sec6-thm2}, which implies $R^T\a\equiv0$ $(\textrm{mod }2)$. Moreover, for any $\a'\in\supp(g_j)$, we have $(\a+\a')_2\in\overline{S}$ and for any $\a''\in\N^n_{2\hat{d}}\backslash\supp(G)$, we have $(\a'')_2\notin\overline{S}$ by Claim \uppercase\expandafter{\romannumeral2} in the proof of Theorem \ref{sec6-thm2} and hence $(\a''+\a')_2\notin\overline{S}$. From these facts we deduce that substituting $t_i$ by $s_i$ in \eqref{sec7-eq1} is just removing the terms whose exponents modulo $2$ are not in $\overline{S}$ from the right hand side of \eqref{sec7-eq1}. Doing so, one does not change the match of coefficients on both sides of the equality. Thus we have
\begin{equation*}
f=s_0+\sum_{j=1}^ms_jg_j,
\end{equation*}
with the desired property.
\end{proof}
\section{Numerical experiments}
\label{sec:benchs}
In this section, we present numerical results of the proposed primal-dual hierarchies \eqref{sec4-eq1}-\eqref{sec4-eq3} and \eqref{sec-eq1}-\eqref{sec6-eq1} of block SDP relaxations for both unconstrained and constrained polynomial optimization problems, respectively.
Our algorithm, named TSSOS, is implemented in Julia for constructing instances of the dual SDP problems \eqref{sec4-eq1}  and \eqref{sec6-eq1}, then relies on MOSEK \cite{mosek} to solve them. 
\newtssos{
TSSOS utilizes the Julia packages LightGraphs \cite{graph} to handle graphs and JuMP \cite{jump} to model SDP. 
}
In the following subsections, we compare the performance of TSSOS with that of GloptiPoly \cite{he} and Yalmip \cite{lo}.
As for TSSOS, GloptiPoly and Yalmip also rely on MOSEK to solve SDP problems. 

Our TSSOS tool can be downloaded at  \href{https://github.com/wangjie212/TSSOS.}{github:TSSOS}.
%
All numerical examples were computed on an Intel Core i5-8265U@1.60GHz CPU with 8GB RAM memory and the WINDOWS 10 system. \newtssos{The timing includes the time for pre-processing (to get the block-structure in TSSOS), the time for modeling SDP and the time for solving SDP. 
Although the modeling part in Julia is usually faster than the one in Matlab, typically the time for solving SDP is dominant on the tested examples in this paper and exceeds the pre-processing time and the modeling time by one order of magnitude.}

The notations that we use are listed in Table \ref{table1}.
\begin{table}[htbp]
\caption{The notations}\label{table1}
\begin{center}
\begin{tabular}{|c|c|}
\hline
$n$&the number of variables\\
\hline
$2d$&the degree\\
\hline
$s$&the number of terms\\
\hline
$\hat{d}$&the relaxation order of Lasserre hierarchy\\
\hline
$k$&the sparse order of the TSSOS hierarchy\\
\hline
bs&the size of monomial bases\\
\hline
\multirow{3}*{mb}&the maximal size of blocks (or a vector whose\\ 
&$k$-th entry is the maximal size of blocks obtained from the\\ 
&TSSOS hierarchy at sparse order $k$ in Table \ref{tb:randpoly1} and Table \ref{tb:randpoly2})\\
\hline
\multirow{3}*{opt}&the optimal value (or a vector\\ 
&whose $k$-th entry is the optimal value obtained from the\\ 
&TSSOS hierarchy at sparse order $k$ in Table \ref{tb:randpoly1} and Table \ref{tb:randpoly2})\\
\hline
\multirow{3}*{time}&running time in seconds\\ 
&(or a vector whose $k$-th entry is the time for computing the\\ 
&TSSOS hierarchy at sparse order $k$ in Table \ref{tb:randpoly1} and Table \ref{tb:randpoly2})\\
\hline
$0$&a number whose absolute value less than $\num{1e-5}$\\
\hline
\#block&the size of blocks\\
\hline
$i\times j$&$j$ blocks of size $i$\\
\hline
-&out of memory\\
\hline
\end{tabular}
\end{center}
\end{table}

\subsection{Unconstrained polynomial optimization problems}
For the unconstrained case, let us first look at an illustrative example.
\begin{example}
Let
\begin{align*}
    f=\,&4(\sum_{i=1}^4p_i^2)^4\sum_{i=1}^4p_i^2a_i^{10}-(\sum_{i=1}^4p_i^2)^3\sum_{i=1}^4p_i^2a_i^{8}\sum_{i=1}^4p_i^2a_i^{2}-(\sum_{i=1}^4p_i^2a_i^{2})^5\\
    &+2(\sum_{i=1}^4p_i^2)^2\sum_{i=1}^4p_i^2a_i^{6}(\sum_{i=1}^4p_i^2a_i^{2})^2-3(\sum_{i=1}^4p_i^2)^2(\sum_{i=1}^4p_i^2a_i^{4})^2\sum_{i=1}^4p_i^2a_i^{2}\\
    &+3\sum_{i=1}^4p_i^2\sum_{i=1}^4p_i^2a_i^{4}(\sum_{i=1}^4p_i^2a_i^{2})^3-4(\sum_{i=1}^4p_i^2)^3\sum_{i=1}^4p_i^2a_i^{6}\sum_{i=1}^4p_i^2a_i^{4}.
\end{align*}
The polynomial $f$ has $8$ variables and is of degree $20$. We compute a basis by the Newton polytope method \eqref{sec2-eq3} which has $1284$ monomials. The first step of the TSSOS hierarchy gives us a block-structure as follows:
\begin{center}
\begin{tabular}{|c|c|c|c|c|c|c|c|c|c|c|c|}
\hline
size&$1$&$2$&$3$&$4$&$10$&$11$&$14$&$19$&$20$&$31$&$42$\\
\hline
number&$1$&$6$&$36$&$18$&$5$&$6$&$4$&$1$&$18$&$12$&$4$\\
\hline
\end{tabular}
\end{center}
where the first line is the size of blocks and the second line is the number of blocks of the corresponding size. We obtain the optimal value $\num{-2.1617e-6}$ at the first step of the TSSOS hierarchy. The whole computation takes only $12$s! It turns out that the hierarchy converges at the first iteration for this polynomial.
\end{example}

\bigskip
\noindent{\bf{\large\textbullet \,\,Randomly generated examples}}
\bigskip

Now we present the numerical results for randomly generated polynomials of two types. The first type is of the SOS form. More concretely, we consider the polynomial $$f=\sum_{i=1}^tf_i^2\in\textbf{randpoly1}(n,2d,t,p) \,,$$ constructed as follows: first randomly choose a subset of monomials $M$ from $\x^{\N^n_{d}}$ with probability $p$, and then randomly assign the elements of $M$ to $f_1,\ldots,f_t$ with random coefficients between $-1$ and $1$. We generate $18$ random polynomials $F_1,\ldots,F_{18}$ from $6$ different classes\footnote{The polynomials can be downloaded at https://wangjie212.github.io/jiewang/code.html.}, where $$F_1,F_2,F_3\in\textbf{randpoly1}(8,8,30,0.1),$$ $$F_4,F_5,F_6\in\textbf{randpoly1}(8,10,25,0.04),$$
$$F_{7},F_{8},F_{9}\in\textbf{randpoly1}(9,10,30,0.03),$$
$$F_{10},F_{11},F_{12}\in\textbf{randpoly1}(10,12,20,0.01),$$
$$F_{13},F_{14},F_{15}\in\textbf{randpoly1}(10,16,30,0.003),$$
$$F_{16},F_{17},F_{18}\in\textbf{randpoly1}(12,12,50,0.01).$$
For these polynomials, the sign-symmetry is always trivial. We compute a monomial basis using the Newton polytope method \eqref{sec2-eq3}.
Table \ref{tb:randpoly1} displays the numerical results on these polynomials.
Note that the time for computing a monomial basis is included in the time of the first step of the TSSOS hierarchy.
In Table \ref{tb:randpoly1c}, we compare the performance of TSSOS ($k=1$), GloptiPoly and Yalmip on these polynomials.
In Yalmip, we turn the option ``sos.newton" on to compute a monomial basis also by the Newton polytope method.

For these examples, TSSOS always provides a nice block-structure at sparse order $k=1$ and retrieves the same optimum as the dense moment-SOS relaxation in much less time. TSSOS is also significantly faster than Yalmip. Due to the memory limit, GloptiPoly (resp. Yalmip) cannot handle polynomials with more than $8$ (resp. $10$) variables while TSSOS can solve problems involving up to $12$ variables.

\begin{table}[htbp]
\caption{The results for randomly generated polynomials of type \uppercase\expandafter{\romannumeral1}}\label{tb:randpoly1}
\begin{center}
\begin{tabular}{|c|c|c|c|c|c|c|c|}
\hline
&$n$&$2d$&$s$&bs&mb&opt&time\\
\hline
$F_1$&$8$&$8$&$64$&$106$&$[31,105,106]$&$[0,0,0]$&$[1.7,3.8,3.9]$\\
\hline
$F_2$&$8$&$8$&$102$&$122$&$[71,122]$&$[0,0]$&$[4.6,11]$\\
\hline
$F_3$&$8$&$8$&$104$&$150$&$[102,150]$&$[0,0]$&$[8.8,15]$\\
\hline
$F_4$&$8$&$10$&$103$&$202$&$[64,202]$&$[0,0]$&$[4.8,83]$\\
\hline
$F_5$&$8$&$10$&$85$&$201$&$[66,201]$&$[0,0]$&$[4.2,68]$\\
\hline
$F_6$&$8$&$10$&$111$&$128$&$[76,128]$&$[0,0]$&$[5.2,20]$\\
\hline
$F_{7}$&$9$&$10$&$101$&$145$&$[35,142,145]$&$[0,0,0]$&$[3.2,38,42]$\\
\hline
$F_{8}$&$9$&$10$&$166$&$178$&$[67,178]$&$[0,0]$&$[6.5,96]$\\
\hline
$F_{9}$&$9$&$10$&$161$&$171$&$[62,170,171]$&$[0,0,0]$&$[5.9,89,101]$\\
\hline
$F_{10}$&$10$&$12$&$271$&$223$&$[75,220,223]$&$[0,0,0]$&$[12,403,435]$\\
\hline
$F_{11}$&$10$&$12$&$253$&$176$&$[60,167,176]$&$[0,0,0]$&$[9.2,98,122]$\\
\hline
$F_{12}$&$10$&$12$&$261$&$204$&$[73,204]$&$[0,0]$&$[12,324]$\\
\hline
$F_{13}$&$10$&$16$&$370$&$1098$&$[99,1098]$&$[0,\textrm{-}]$&$[36,\textrm{-}]$\\
\hline
$F_{14}$&$10$&$16$&$412$&$800$&$[195,800]$&$[0,\textrm{-}]$&$[305,\textrm{-}]$\\
\hline
$F_{15}$&$10$&$16$&$436$&$618$&$[186,617,618]$&$[0,\textrm{-},\textrm{-}]$&$[207,\textrm{-},\textrm{-}]$\\
\hline
$F_{16}$&$12$&$12$&$488$&$330$&$[129,324,330]$&$[0,\textrm{-},\textrm{-}]$&$[61,\textrm{-},\textrm{-}]$\\
\hline
$F_{17}$&$12$&$12$&$351$&$264$&$[26,42,151,263,264]$&$[0,0,0,\textrm{-},\textrm{-}]$&$[17,0.45,76,\textrm{-},\textrm{-}]$\\
 \hline
$F_{18}$&$12$&$12$&$464$&$316$&$[45,274,316]$&$[0,\textrm{-},\textrm{-}]$&$[22,\textrm{-},\textrm{-}]$\\
\hline
\end{tabular}
\end{center}
\end{table}

\begin{table}[htbp]
\caption{Comparison with GloptiPoly and Yalmip for randomly generated polynomials of type \uppercase\expandafter{\romannumeral1}}\label{tb:randpoly1c}
\begin{center}
\begin{tabular}{|c|c|c|c||c|c|c|c|}
\hline
&\multicolumn{3}{c||}{time}&&\multicolumn{3}{c|}{time}\\
\hline
&TSSOS&GloptiPoly&Yalmip&&TSSOS&GloptiPoly&Yalmip\\
\hline
$F_1$&$1.7$&$306$&$4.9$&$F_{10}$&$12$&-&$474$\\
\hline
$F_2$&$4.6$&$348$&$13$&$F_{11}$&$9.2$&-&$147$\\
\hline
$F_3$&$8.8$&$326$&$19$&$F_{12}$&$12$&-&$350$\\
\hline
$F_4$&$4.8$&-&$92$&$F_{13}$&$36$&-&-\\
\hline
$F_5$&$4.2$&-&$72$&$F_{14}$&$305$&-&-\\
\hline
$F_6$&$5.2$&-&$22$&$F_{15}$&$207$&-&-\\
\hline
$F_7$&$3.2$&-&$44$&$F_{16}$&$61$&-&-\\
\hline
$F_8$&$6.5$&-&$143$&$F_{17}$&$17$&-&-\\
\hline
$F_9$&$5.9$&-&$109$&$F_{18}$&$22$&-&-\\
\hline
\end{tabular}
\end{center}
\end{table}

The second type of randomly generated problems are polynomials whose Newton polytopes are scaled standard simplices.
More concretely, we consider polynomials defined by $$f=c_0+\sum_{i=1}^nc_ix_i^{2d}+\sum_{j=1}^{s-n-1}c_j'\x^{\a_j}\in\textbf{randpoly2}(n,2d,s) \,,$$ constructed as follows: we randomly choose coefficients $c_i$ between $0$ and $1$, as well as $s-n-1$ vectors $\a_j$ in $\N^n_{2d-1}\backslash\{\mathbf{0}\}$ with random coefficients $c_j'$ between $-1$ and $1$. We generate $18$ random polynomials $G_1,\ldots,G_{18}$ from $6$ different classes\footnote{The polynomials can be downloaded at https://wangjie212.github.io/jiewang/code.html.}, where $$G_1,G_2,G_3\in\textbf{randpoly2}(8,8,15),$$ $$G_4,G_5,G_6\in\textbf{randpoly2}(9,8,20),$$
$$G_7,G_8,G_9\in\textbf{randpoly2}(9,10,15),$$
$$G_{10},G_{11},G_{12}\in\textbf{randpoly2}(10,8,20),$$
$$G_{13},G_{14},G_{15}\in\textbf{randpoly2}(11,8,20),$$
$$G_{16},G_{17},G_{18}\in\textbf{randpoly2}(12,8,25).$$
Table \ref{tb:randpoly2} displays the numerical results on these polynomials.
Table \ref{tb:randpoly2c} indicates similar efficiency and accuracy results on the comparison with GloptiPoly and Yalmip as for randomly generated polynomials of type I.
In Yalmip, we turn the option ``sos.congruence" on to take sign-symmetries into account, which allows one to handle slightly more polynomials than GloptiPoly.

\begin{table}[htbp]
\caption{The results for randomly generated polynomials of type \uppercase\expandafter{\romannumeral2}}\label{tb:randpoly2}
\begin{center}
\begin{tabular}{|c|c|c|c|c|c|c|c|}
\hline
&$n$&$2d$&$s$&bs&mb&opt&time\\
\hline
$G_1$&$8$&$8$&$15$&$495$&$[126,219]$&$[-0.5758,-0.5758]$&$[8.5,26]$\\
\hline
$G_2$&$8$&$8$&$15$&$495$&$[86,169]$&$[-34.6897,-34.6897]$&$[2.6,21]$\\
\hline
$G_3$&$8$&$8$&$15$&$495$&$[59,75]$&$[0.7073,0.7073]$&$[1.0,3.3]$\\
\hline
$G_4$&$9$&$8$&$20$&$715$&$[170,715]$&$[-801.6920,\textrm{-}]$&$[40,\textrm{-}]$\\
\hline
$G_5$&$9$&$8$&$20$&$715$&$[160,365]$&$[-0.8064,-0.8064]$&$[24,322]$\\
\hline
$G_6$&$9$&$8$&$20$&$715$&$[186,331]$&$[-1.6981,-1.6981]$&$[31,126]$\\
\hline
$G_{7}$&$9$&$10$&$15$&$2002$&$[122,224]$&$[-1.2945,-1.2945]$&$[24,303]$\\
\hline
$G_{8}$&$9$&$10$&$15$&$2002$&$[143,170]$&$[-0.6622,-0.6622]$&$[28,195]$\\
\hline
$G_{9}$&$9$&$10$&$15$&$2002$&$[154,208]$&$[0.5180,0.5180]$&$[21,180]$\\
\hline
$G_{10}$&$10$&$8$&$20$&$1001$&$[133,525]$&$[-0.4895,\textrm{-}]$&$[13,\textrm{-}]$\\
\hline
$G_{11}$&$10$&$8$&$20$&$1001$&$[223,403]$&$[0.1867,0.1867]$&$[86,481]$\\
\hline
$G_{12}$&$10$&$8$&$20$&$1001$&$[208,511]$&$[0.4943,\textrm{-}]$&$[66,\textrm{-}]$\\
\hline
$G_{13}$&$11$&$8$&$20$&$1365$&$[110,296]$&$[-3.9625,-3.9625]$&$[13,580]$\\
\hline
$G_{14}$&$11$&$8$&$20$&$1365$&$[128,436]$&$[-2.1835,\textrm{-}]$&$[37,\textrm{-}]$\\
\hline
$G_{15}$&$11$&$8$&$20$&$1365$&$[174,272]$&$[0.0588,0.0588]$&$[36,310]$\\
\hline
$G_{16}$&$12$&$8$&$25$&$1820$&$[263,924]$&$[-688.0269,\textrm{-}]$&$[693,\textrm{-}]$\\
\hline
$G_{17}$&$12$&$8$&$25$&$1820$&$[256,924]$&$[-40.2178,\textrm{-}]$&$[333,\textrm{-}]$\\
\hline
$G_{18}$&$12$&$8$&$25$&$1820$&$[275,924]$&$[-14.2693,\textrm{-}]$&$[393,\textrm{-}]$\\
\hline
\end{tabular}
\end{center}
\end{table}

\begin{table}[htbp]
\caption{Comparison with GloptiPoly and Yalmip for randomly generated polynomials of type \uppercase\expandafter{\romannumeral2}}\label{tb:randpoly2c}
\begin{center}
\begin{tabular}{|c|c|c|c|c|c|c|}
\hline
&\multicolumn{2}{c|}{TSSOS}&\multicolumn{2}{c|}{GloptiPoly}&\multicolumn{2}{c|}{Yalmip}\\
\cline{2-7}
&opt&time&opt&time (s)&opt&time\\
\hline
$G_1$&$-0.5758$&$8.5$&$-0.5758$&$346$&$-0.5758$&$31$\\
\hline
$G_2$&$-34.6897$&$2.6$&$-34.690$&$447$&$-34.6897$&$24$\\
\hline
$G_3$&$0.7073$&$1.0$&$0.7073$&$257$&$0.7073$&$6.0$\\
\hline
$G_4$&$-801.692$&$40$&-&-&-&-\\
\hline
$G_5$&$-0.8064$&$24$&-&-&$-0.8064$&$363$\\
\hline
$G_6$&$-1.6981$&$31$&-&-&$-1.6981$&$141$\\
\hline
$G_7$&$-1.2945$&$24$&-&-&$-1.2945$&$322$\\
\hline
$G_8$&$-0.6622$&$28$&-&-&$-0.6622$&$233$\\
\hline
$G_9$&$0.5180$&$21$&-&-&$0.5180$&$249$\\
\hline
$G_{10}$&$-0.4895$&$13$&-&-&-&-\\
\hline
$G_{11}$&$0.1867$&$86$&-&-&$0.1867$&$536$\\
\hline
$G_{12}$&$0.4943$&$66$&-&-&-&-\\
\hline
$G_{13}$&$-3.9625$&$13$&-&-&$-3.9625$&$655$\\
\hline
$G_{14}$&$-2.1835$&$37$&-&-&-&-\\
\hline
$G_{15}$&$0.0588$&$36$&-&-&$0.0588$&$340$\\
\hline
$G_{16}$&$-688.0269$&$693$&-&-&-&-\\
\hline
$G_{17}$&$-40.2178$&$333$&-&-&-&-\\
\hline
$G_{18}$&$-14.2693$&$393$&-&-&-&-\\
\hline
\end{tabular}
\end{center}
\end{table}

\bigskip
\noindent{\bf{\large\textbullet \,\,Examples from networked systems}}
\bigskip

Next we consider Lyapunov functions emerging from some networked systems. In \cite{han}, the authors propose a structured SOS decomposition for those systems, which allows them to handle structured Lyapunov function candidates up to $50$ variables.

The following polynomial is from Example 2 in \cite{han}:
\begin{equation*}
f=\sum_{i=1}^Na_i(x_i^2+x_i^4)-\sum_{i=1}^N\sum_{k=1}^Nb_{ik}x_i^2x_k^2,
\end{equation*}
where $a_i$ are randomly chosen from $[1,2]$ and $b_{ik}$ are randomly chosen from $[\frac{0.5}{N},\frac{1.5}{N}]$. Here, $N$ is the number of nodes in the network. The task is to determine whether $f$ is globally nonnegative.
Here we solve again SDP \eqref{sec4-eq1} at $k=1$ with TSSOS for $N=10,20,30,40,50,60,70,80$. The results are listed in Table \ref{tb:network1}.
\begin{table}[htbp]
\caption{The results for network problem \uppercase\expandafter{\romannumeral1}}\label{tb:network1}
\begin{center}
\begin{tabular}{|c|c|c|c|c|c|c|c|c|}
\hline
$N$&$10$&$20$&$30$&$40$&$50$&$60$&$70$&$80$\\
\hline
mb&$11$&$31$&$31$&$41$&$51$&$61$&$71$&$81$\\
\hline
time&$0.006$&$0.03$&$0.10$&$0.34$&$0.92$&$1.9$&$4.7$&$12$\\
\hline
\end{tabular}
\end{center}
\end{table}

For this example, the size of systems that can be handled in \cite{han} is up to $N=50$ nodes while our approach can easily handle systems with up to $N=80$ nodes.

The following polynomial is from Example 3 in \cite{han}:
\begin{equation}\label{sec8-eq1}
V=\sum_{i=1}^Na_i(\frac{1}{2}x_i^2-\frac{1}{4}x_i^4)+\frac{1}{2}\sum_{i=1}^N\sum_{k=1}^Nb_{ik}\frac{1}{4}(x_i-x_k)^4,
\end{equation}
where $a_i$ are randomly chosen from $[0.5,1.5]$ and $b_{ik}$ are randomly chosen from $[\frac{0.5}{N},\frac{1.5}{N}]$.
The task is to analyze the domain on which the Hamiltonian function $V$ for a network of Duffing oscillators is positive
definite. We use the following condition to establish an inner approximation of the domain on which $V$ is positive definite:
\begin{equation}\label{sec8-eq2}
f=V-\sum_{i=1}^N\lambda_ix_i^2(g-x_i^2)\ge0 \,,
\end{equation}
where $\lambda_i>0$ are scalar decision variables and $g$ is a fixed positive scalar.
Clearly, the condition \eqref{sec8-eq2} ensures that $V$ is positive definite when $x_i^2<g$.
Here we solve SDP \eqref{sec4-eq1} at $k=1$ with TSSOS for $N=10,20,30,40,50$.
For this example, graphs arising in the TSSOS hierarchy are naturally chordal, so we simply exploit chordal decompositions.
This example was also examined in \cite{maj} to demonstrate the advantage of SDSOS programming compared to dense SOS programming. The method based on SDSOS programming was executed in SPOT \cite{me} with MOSEK as a \newtssos{second-order cone programming solver}. The results are listed in Table \ref{tb:network2}. The row ``\#var" in Table \ref{tb:network2} indicates the number of decision variables.

\begin{table}[htbp]
\newtssos{
\caption{The results for network problem \uppercase\expandafter{\romannumeral2}}\label{tb:network2}
\begin{center}
\begin{tabular}{|c|c|c|c|c|c|c|}
\hline
\multicolumn{2}{|c|}{$N$}&$10$&$20$&$30$&$40$&$50$\\
\hline
\multirow{4}*{\#block}&\multirow{3}*{TSSOS}&$3\times45,$&$3\times190,$&$3\times435,$&$3\times780,$&$3\times1225,$\\
&&$1\times10,$&$1\times20,$&$1\times30,$&$1\times40,$&$1\times50,$\\
&&$11\times1$&$21\times1$&$31\times1$&$41\times1$&$51\times1$\\
\cline{2-7}
&SDSOS&$2\times2145$&$2\times26565$&$2\times122760$&$2\times370230$&$2\times878475$\\
\hline
\multirow{2}*{\#var}&TSSOS&$346$&$1391$&$3136$&$5581$&$8726$\\
\cline{2-7}
&SDSOS&$6435$&$79695$&$368280$&$1110690$&$2635425$\\
\hline
\multirow{2}*{time}&TSSOS&$0.01$&$0.06$&$0.17$&$0.50$&$0.89$\\
\cline{2-7}
&SDSOS&$0.47$&$1.14$&$5.47$&$20$&$70$\\
\hline
\end{tabular}
\end{center}
}
\end{table}

\newtssos{For this example, TSSOS uses much less decision variables than SDSOS programming, and hence spends less time compared to SDSOS programming.} On the other hand,
TSSOS computes a positive definite form $V$ after selecting a value for $g$ up to $2$ (which is the same as the maximal value obtained by the dense SOS) while the method in \cite{han} can select $g$ up to $1.8$ and the one based on SDSOS programming only works out for a maximal value of $g$ up to around $1.5$.

\bigskip
\noindent{\bf{\large\textbullet \,\,Broyden banded functions}}
\bigskip
The Broyden banded function (\cite{waki}) is defined by
\begin{equation*}
    f_{\textrm{Bb}}(\x)=\sum_{i=1}^n(x_i(2+5x_i^2)+1-\sum_{j\in J_i}(1+x_j)x_j)^2,
\end{equation*}
where $J_i=\{j\mid j\ne i, \max(1,i-5)\le j\le\min(n,i+1)\}$.
We prove that $f_{\textrm{Bb}}$ is nonnegative by solving SDP \eqref{sec4-eq1} at $k=1$ with TSSOS for $n=6,7,8,9,10$. We make a comparison between TSSOS and SparsePOP \cite{WakiKKMS08} which exploits correlative sparsity and uses \newtssos{SeDuMi \cite{sedumi} as an SDP solver}. For this example, since TSSOS and SparsePOP use different SDP solvers, the running time is not comparable directly. We thereby also provide the number of decision variables involved in TSSOS and SparsePOP respectively. The results are displayed in Table \ref{tb:bdfunction}. The row ``\#var" in Table \ref{tb:bdfunction} indicates the number of decision variables.

\begin{table}[htbp]
\caption{The results for Broyden banded functions}\label{tb:bdfunction}
\begin{center}
\begin{tabular}{|c|c|c|c|c|c|c|}
\hline
\multicolumn{2}{|c|}{$n$}&$6$&$7$&$8$&$9$&$10$\\
\hline
\multirow{3}*{\#block}&\multirow{2}*{TSSOS}&$64\times1,$&$85\times1,$&$108\times1,$&$133\times1,$&$160\times1,$\\
&&$1\times20$&$1\times35$&$1\times57$&$1\times87$&$1\times126$\\
\cline{2-7}
&SparsePOP&$84\times1$&$120\times1$&$120\times2$&$120\times3$&$120\times4$\\
\hline
\multirow{2}*{\#var}&TSSOS&$2100$&$3690$&$5943$&$8998$&$13006$\\
\cline{2-7}
&SparsePOP&$3570$&$7260$&$14520$&$21780$&$29040$\\
\hline
\multirow{2}*{time}&TSSOS&$0.27$&$0.76$&$1.9$&$5.3$&$13$\\
\cline{2-7}
&SparsePOP&$2.0$&$9.0$&$20$&$30$&$42$\\
\hline
\end{tabular}
\end{center}
\end{table}

\subsection{Constrained polynomial optimization problems}
For the constrained case, we also begin with an illustrative example.
\begin{example}
Consider the following problem:
\begin{equation*}
\begin{cases}
\min\quad &f=27-((x_1-x_2)^2+(y_1-y_2)^2)((x_1-x_3)^2+(y_1-y_3)^2)\\
&\quad\quad((x_2-x_3)^2+(y_2-y_3)^2)\\
\textrm{s.t.}\quad &g_1=((x_1^2+y_1^2)+(x_2^2+y_2^2)+(x_3^2+y_3^2))-3\\
&g_2=3-((x_1^2+y_1^2)+(x_2^2+y_2^2)+(x_3^2+y_3^2))
\end{cases}
\end{equation*}
We consider the TSSOS hierarchy with $\hat{d}=3$ and $\hat{d}=4$.
For $\hat d = 3$, and $k = 1$, we obtain the following block-structure:
\begin{center}
\begin{tabular}{|c|c|}
\hline
$M_3(\y)$&$31\times2,7\times1,1\times15$\\
\hline
$M_{2}(g_1\y)$&$13\times1,9\times1,1\times6$\\
\hline
$M_{2}(g_2\y)$&$13\times1,9\times1,1\times6$\\
\hline
\end{tabular}
\end{center}
and we obtain an optimal value $\num{-5.0324e-8}$.
For $\hat d = 3$, and $k = 2$, we have
\begin{center}
\begin{tabular}{|c|c|}
\hline
$M_3(\y)$&$31\times2,13\times1,9\times1$\\
\hline
$M_{2}(g_1\y)$&$13\times1,9\times1,3\times2$\\
\hline
$M_{2}(g_2\y)$&$13\times1,9\times1,3\times2$\\
\hline
\end{tabular}
\end{center}
and an optimal value $\num{-1.6016e-7}$.
For $\hat d = 3$, the hierarchy converges at $k = 2$.

For $\hat{d}=4$, the hierarchy immediately converges at $k = 1$, yielding the following block-structure:
\begin{center}
\begin{tabular}{|c|c|}
\hline
$M_4(\y)$&$79\times1,69\times1,31\times2$\\
\hline
$M_{3}(g_1\y)$&$31\times2,13\times1,9\times1$\\
\hline
$M_{3}(g_2\y)$&$31\times2,13\times1,9\times1$\\
\hline
\end{tabular}
\end{center}
and an optimal value $\num{-2.5791e-10}$.
\end{example}

Now we present the numerical results for constrained polynomial optimization problems. We generate six randomly generated polynomials $H_1,\ldots,H_6$ of type \uppercase\expandafter{\romannumeral2}\footnote{The polynomials can be downloaded at https://wangjie212.github.io/jiewang/code.html.} as objective functions $f$ and minimize $f$ over a basic semialgebraic set $\mathbf{K}\subseteq\R^{n}$ for two cases: the unit ball $$\mathbf{K}=\{(x_1,\ldots,x_n)\in\R^n\mid g_1=1-(x_1^2+\cdots+x_n^2)\ge0\} \,,$$ and the unit hypercube $$\mathbf{K}=\{(x_1,\ldots,x_n)\in\R^n\mid g_1=1-x_1^2\ge0,\ldots,g_n=1-x_n^2\ge0\}.$$

We compare the performance of TSSOS and GloptiPoly in these two cases.
The related numerical results are outputted in Table \ref{unitball} and Table \ref{unithypercube}.
As in the unconstrained case, Table \ref{unitball} and Table \ref{unithypercube} show that TSSOS performs much better than the dense moment-SOS without compromising accuracy.

\begin{table}[htbp]
\caption{The results for minimizing randomly generated polynomials of type \uppercase\expandafter{\romannumeral2} over unit balls}\label{unitball}
\begin{center}
\begin{tabular}{|c|c|c|c|c|c|c|c|c|c|}
\hline
&\multirow{2}*{$(n,2d,s)$}&\multirow{2}*{$\hat{d}$}&\multirow{2}*{$k$}&\multirow{2}*{mb}&\multicolumn{2}{c|}{TSSOS}&\multicolumn{2}{c|}{GloptiPoly}\\
\cline{6-9}
&&&&&opt&time&opt&time\\
\hline
\multirow{4}*{$H_1$}&$\multirow{4}*{(6,8,10)}$&\multirow{2}*{$4$}&$1$&$(59,25)$&$0.1362$&$0.67$&\multirow{2}*{$0.1362$}&\multirow{2}*{$8.0$}\\
\cline{4-7}
&&&$2$&$(59,25)$&$0.1362$&$0.39$&&\\
\cline{3-9}
&&\multirow{2}*{$5$}&$1$&$(113,59)$&$0.1362$&$3.0$&\multirow{2}*{$0.1362$}&\multirow{2}*{$80$}\\
\cline{4-7}
&&&$2$&$(113,59)$&$0.1362$&$3.1$&&\\
\hline
\multirow{4}*{$H_2$}&$\multirow{4}*{(7,8,12)}$&\multirow{2}*{$4$}&$1$&$(85,36)$&$0.1373$&$1.6$&\multirow{2}*{$0.1373$}&\multirow{2}*{$34$}\\
\cline{4-7}
&&&$2$&$(99,40)$&$0.1373$&$1.7$&&\\
\cline{3-9}
&&\multirow{2}*{$5$}&$1$&$(176,85)$&$0.1373$&$11$&\multirow{2}*{-}&\multirow{2}*{-}\\
\cline{4-7}
&&&$2$&$(212,99)$&$0.1373$&$21$&&\\
\hline
\multirow{4}*{$H_3$}&$\multirow{4}*{(8,8,15)}$&\multirow{2}*{$4$}&$1$&$(69,23)$&$0.1212$&$2.8$&\multirow{2}*{$0.1212$}&\multirow{2}*{$225$}\\
\cline{4-7}
&&&$2$&$(135,45)$&$0.1212$&$13$&&\\
\cline{3-9}
&&\multirow{2}*{$5$}&$1$&$(144,69)$&$0.1212$&$35$&\multirow{2}*{-}&\multirow{2}*{-}\\
\cline{4-7}
&&&$2$&$(333,135)$&0.1212&425&&\\
\hline
\multirow{4}*{$H_4$}&$\multirow{4}*{(9,6,15)}$&\multirow{2}*{$3$}&$1$&$(48,17)$&$0.8704$&$1.0$&\multirow{2}*{$0.8704$}&\multirow{2}*{$16$}\\
\cline{4-7}
&&&$2$&$(50,17)$&$0.8704$&$0.35$&&\\
\cline{3-9}
&&\multirow{2}*{$4$}&$1$&$(131,48)$&$0.8704$&$6.8$&\multirow{2}*{-}&\multirow{2}*{-}\\
\cline{4-7}
&&&$2$&$(140,50)$&$0.8704$&$9.7$&&\\
\hline
\multirow{4}*{$H_5$}&$\multirow{4}*{(10,6,20)}$&\multirow{2}*{$3$}&$1$&$(67,22)$&$0.5966$&$2.1$&\multirow{2}*{$0.5966$}&\multirow{2}*{$48$}\\
\cline{4-7}
&&&$2$&$(92,27)$&$0.5966$&$1.6$&&\\
\cline{3-9}
&&\multirow{2}*{$4$}&$1$&$(193,67)$&$0.5966$&$48$&\multirow{2}*{-}&\multirow{2}*{-}\\
\cline{4-7}
&&&$2$&$(274,92)$&$0.5966$&$77$&&\\
\hline
\multirow{4}*{$H_6$}&$\multirow{4}*{(11,6,20)}$&\multirow{2}*{$3$}&$1$&$(67,19)$&$0.1171$&$2.1$&\multirow{2}*{$0.1171$}&\multirow{2}*{$115$}\\
\cline{4-7}
&&&$2$&$(104,28)$&$0.1171$&$4.0$&&\\
\cline{3-9}
&&\multirow{2}*{$4$}&$1$&$(170,67)$&$0.1171$&$40$&\multirow{2}*{-}&\multirow{2}*{-}\\
\cline{4-7}
&&&$2$&$(356,104)$&0.1171&389&&\\
\hline
\end{tabular}\\
{\small In this table, the first entry of mb is the maximal size of blocks corresponding to the moment matrix $M_{\hat{d}}(\y)$ and the second entry of mb is the maximal size of blocks corresponding to the localizing matrix $M_{\hat{d}-d_1}(g_1\y)$.}
\end{center}
\end{table}

\begin{table}[htbp]
\caption{The results for minimizing randomly generated polynomials of type \uppercase\expandafter{\romannumeral2} over unit hypercubes}\label{unithypercube}
\begin{center}
\begin{tabular}{|c|c|c|c|c|c|c|c|c|c|}
\hline
&\multirow{2}*{$(n,2d,s)$}&\multirow{2}*{$\hat{d}$}&\multirow{2}*{$k$}&\multirow{2}*{mb}&\multicolumn{2}{c|}{TSSOS}&\multicolumn{2}{c|}{GloptiPoly}\\
\cline{6-9}
&&&&&opt&time&opt&time\\
\hline
\multirow{4}*{$H_1$}&$\multirow{4}*{(6,8,10)}$&\multirow{2}*{$4$}&$1$&$(59,25)$&$-0.4400$&$1.1$&\multirow{2}*{$-0.4400$}&\multirow{2}*{$19$}\\
\cline{4-7}
&&&$2$&$(59,25)$&$-0.4400$&$0.88$&&\\
\cline{3-9}
&&\multirow{2}*{$5$}&$1$&$(113,59)$&$-0.4400$&$8.0$&\multirow{2}*{$-0.4400$}&\multirow{2}*{$237$}\\
\cline{4-7}
&&&$2$&$(113,59)$&$-0.4400$&$9.1$&&\\
\hline
\multirow{4}*{$H_2$}&$\multirow{4}*{(7,8,12)}$&\multirow{2}*{$4$}&$1$&$(85,34)$&$-0.1289$&$3.0$&\multirow{2}*{$-0.1289$}&\multirow{2}*{$101$}\\
\cline{4-7}
&&&$2$&$(99,40)$&$-0.1289$&$4.1$&&\\
\cline{3-9}
&&\multirow{2}*{$5$}&$1$&$(176,85)$&$-0.1289$&$40$&\multirow{2}*{-}&\multirow{2}*{-}\\
\cline{4-7}
&&&$2$&$(212,99)$&$-0.1289$&$87$&&\\
\hline
\multirow{4}*{$H_3$}&$\multirow{4}*{(8,8,15)}$&\multirow{2}*{$4$}&$1$&$(69,23)$&$-0.1465$&$3.9$&\multirow{2}*{$-0.1465$}&\multirow{2}*{$433$}\\
\cline{4-7}
&&&$2$&$(135,45)$&$-0.1465$&$30$&&\\
\cline{3-9}
&&\multirow{2}*{$5$}&$1$&$(144,69)$&$-0.1465$&$77$&\multirow{2}*{-}&\multirow{2}*{-}\\
\cline{4-7}
&&&$2$&$(333,135)$&$-0.1465$&900&&\\
\hline
\multirow{4}*{$H_4$}&$\multirow{4}*{(9,6,15)}$&\multirow{2}*{$3$}&$1$&$(48,10)$&$0.1199$&$1.3$&\multirow{2}*{$0.1199$}&\multirow{2}*{$27$}\\
\cline{4-7}
&&&$2$&$(50,17)$&$0.1199$&$0.64$&&\\
\cline{3-9}
&&\multirow{2}*{$4$}&$1$&$(131,48)$&$0.1199$&$12$&\multirow{2}*{-}&\multirow{2}*{-}\\
\cline{4-7}
&&&$2$&$(140,50)$&$0.1199$&$26$&&\\
\hline
\multirow{5}*{$H_5$}&$\multirow{5}*{(10,6,20)}$&\multirow{3}*{$3$}&$1$&$(67,13)$&$-0.2813$&$2.1$&\multirow{3}*{$-0.2813$}&\multirow{3}*{$69$}\\
\cline{4-7}
&&&$2$&$(92,27)$&$-0.2813$&$2.7$&&\\
\cline{4-7}
&&&$3$&$(92,27)$&$-0.2813$&$2.7$&&\\
\cline{3-9}
&&\multirow{2}*{$4$}&$1$&$(193,67)$&$-0.2813$&$75$&\multirow{2}*{-}&\multirow{2}*{-}\\
\cline{4-7}
&&&$2$&$(274,92)$&$-0.2813$&$181$&&\\
\hline
\multirow{5}*{$H_6$}&$\multirow{5}*{(11,6,20)}$&\multirow{3}*{$3$}&$1$&$(67,15)$&$-0.2316$&$2.6$&\multirow{3}*{$-0.2316$}&\multirow{3}*{$211$}\\
\cline{4-7}
&&&$2$&$(104,28)$&$-0.2316$&$7.5$&&\\
\cline{4-7}
&&&$3$&$((104,28)$&$-0.2316$&$7.6$&&\\
\cline{3-9}
&&\multirow{2}*{$4$}&$1$&$(170,67)$&$-0.2316$&$103$&\multirow{2}*{-}&\multirow{2}*{-}\\
\cline{4-7}
&&&$2$&$(356,104)$&$-0.2316$&1108&&\\
\hline
\end{tabular}\\
{\small In this table, the first entry of mb is the maximal size of blocks corresponding to the moment matrix $M_{\hat{d}}(\y)$ and the second entry of mb is the maximal size of blocks corresponding to the localizing matrices $M_{\hat{d}-d_j}(g_j\y),j=1,\ldots,m$.}
\end{center}
\end{table}

\section{Conclusions}\label{cons}
We have provided a new variant of the moment-SOS hierarchy to handle polynomial optimization problems with term sparsity.
This hierarchy shares the same theoretical convergence guarantees with the standard one and
our numerical benchmarks demonstrate the performance speedup which can be achieved in both unconstrained and constrained cases.

One direction of further research is to investigate if one can benefit from the same term sparsity exploitation for other variants of the moment-SOS hierarchy, including the ones dedicated to optimal control, approximations of sets of interest (maximal invariant, reachable set) in dynamical systems, or the ones dedicated to eigenvalue and trace optimization of polynomials in non-commuting variables.

\section*{Acknowledgments}
The first and second author were supported from the Tremplin ERC Stg Grant ANR-18-ERC2-0004-01 (T-COPS project).
The second author was supported by the FMJH Program PGMO (EPICS project) and  EDF, Thales, Orange et Criteo.
The second and third author received funding from ANITI, coordinated by the Federal University of Toulouse within the framework of French Program ``Investing for the Future ¨C PIA3'' program under the Grant agreement n$^{\circ}$ANR-19-XXXX-000X.
This work has been supported by European Union’s Horizon 2020 research and innovation programme under the Marie Sklodowska-Curie Actions, grant agreement 813211 (POEMA). 
The research of the third author was funded by the European Research Council (ERC) under the European’s Union Horizon 2020 research and innovation program (grant agreement 666981 TAMING. 

\bibliographystyle{siamplain}

\end{document}

%% file: tssos_share.tex
\usepackage{lipsum}
\usepackage{amsfonts}
\usepackage{graphicx}
\usepackage{epstopdf}
\usepackage{algorithmic}
\ifpdf
  \DeclareGraphicsExtensions{.eps,.pdf,.png,.jpg}
\else
  \DeclareGraphicsExtensions{.eps}
\fi

\newsiamremark{remark}{Remark}
\newsiamremark{hypothesis}{Hypothesis}
\crefname{hypothesis}{Hypothesis}{Hypotheses}
\newsiamthm{claim}{Claim}

\headers{TSSOS: A Moment-SOS hierarchy that exploits term sparsity}{Jie Wang, Victor Magron and Jean-Bernard Lasserre}

\title{TSSOS: A Moment-SOS hierarchy that exploits term sparsity\thanks{Submitted to the editors DATE.
}}

\author{Jie Wang\thanks{Laboratoire d'Analyse et d'Architecture des Syst\`emes (LAAS), Toulouse, France 
  (\email{jwang@laas.fr})}
\and Victor Magron\thanks{Laboratoire d'Analyse et d'Architecture des Syst\`emes (LAAS), Institute of Mathematics, 
University of Toulouse, France 
  (\email{vmagron@laas.fr})}
\and Jean-Bernard Lasserre\thanks{Laboratoire d'Analyse et d’Architecture des Syst\`emes (LAAS), Institute of Mathematics, 
University of Toulouse, France
(\email{lasserre@laas.fr})}}

\usepackage{amsopn}


%% file: tssos_siam.bbl
\begin{thebibliography}{99}

\bibitem{ahmadi2014dsos}
A. A. Ahmadi and A. Majumdar, {\em DSOS and SDSOS optimization: LP and SOCP-based alternatives to sum of squares optimization}, 48th annual conference on information sciences and systems (CISS), 2014:1-5.

\bibitem{graph}
S. Bromberger, J. Fairbanks, and other contributors,
\newblock JuliaGraphs/LightGraphs.jl: an optimized graphs package for the Julia programming language
\newblock \href{https://doi.org/10.5281/zenodo.889971}{doi:10.5281/zenodo.889971}, 2017.

\bibitem{Chandrasekaran:Shah:SAGE}
V. Chandrasekaran and P. Shah, {\em Relative Entropy Relaxations for Signomial
  Optimization.}
  \newblock \emph{SIAM J. Optim.} 26(2):1147--1173, 2016.

\bibitem{re2}
M. D. Choi, T. Y. Lam, and B. Reznick, {\em Sums of squares of real polynomials}, Proceedings of Symposia in Pure mathematics, AMS, 58(1995): 103-126.

\bibitem{jump}
I. Dunning, J. Huchette, and M. Lubin, 
\newblock JuMP: A modeling language for mathematical optimization, \newblock SIAM Review, 59(2):295-320, 2017.

\bibitem{fukuda2001exploiting}
M. Fukuda, M. Kojima, K. Murota, and K. Nakata.
\newblock Exploiting sparsity in semidefinite programming via matrix
  completion. {I}. {G}eneral framework.
\newblock {\em SIAM J. Optim.}, 11(3):647--674, 2000/01.

\bibitem{Ghasemi:Marshall:GP}
M. Ghasemi and M. Marshall, {\em Lower bounds for polynomials using geometric programming},
\newblock \emph{SIAM J. Optim.}, 22(2):460--473, 2012.

\bibitem{han}
E. J. Hancock and A. Papachristodoulou, {\em Structured Sum of Squares for Networked Systems Analysis}, 50th IEEE Conference on Decision and Control and European Control Conference (CDC-ECC).

\bibitem{he}
D. Henrion and J. B. Lasserre, {\em GloptiPoly: Global Optimization over Polynomials with Matlab and SeDuMi}, IEEE Conf. Decis. Control, Las Vegas, Nevada, 2002:747-752.

\bibitem{Iliman:deWolff:Circuits}
S. Iliman and T. de Wolff, {\em Amoebas, nonnegative polynomials and sums of
  squares supported on circuits}, Res. Math. Sci. 3:3-9, 2016.

\bibitem{Josz16}
C\'{e}dric Josz.
\newblock {\em {Application of polynomial optimization to electricity
  transmission networks}}.
\newblock Theses, {Universit{\'e} Pierre et Marie Curie - Paris VI}, July 2016.

\bibitem{ncsparse}
I. Klep, V. Magron, J. Povh, {\em Sparse Noncommutative Polynomial Optimization}.
\newblock {\em preprint   \href{http://arxiv.org/abs/1909.00569}{arXiv:1909.00569}}, 2019.

\bibitem{ko}
M. Kojima, S. Kim, H. Waki, {\em Sparsity in sums of squares of polynomials}, Math. Program., 103(2005):45-62.

\bibitem{las1}
J. B. Lasserre, {\em Global optimization with polynomials and the problem of moments}, SIAM Journal on Optimization, 11(3)(2001):796-817.

\bibitem{Las06}
J.-B. Lasserre.
\newblock Convergent {SDP}-relaxations in polynomial optimization with
  sparsity.
\newblock {\em SIAM J. Optim.}, 17(3):822--843, 2006.

\bibitem{TohLasserre}
J.-B. Lasserre, K.-C. Toh, and S. Yang.
\newblock A bounded degree {SOS} hierarchy for polynomial optimization.
\newblock {\em EURO J. Comput. Optim.}, 5(1-2):87--117, 2017.

\bibitem{Lau09b}
M. Laurent.
\newblock Sums of squares, moment matrices and optimization over polynomials.
\newblock In {\em Emerging applications of algebraic geometry}, volume 149 of
  {\em IMA Vol. Math. Appl.}, pages 157--270. Springer, New York, 2009.

\bibitem{lo}
J. L\"ofberg, YALMIP: a toolbox for modeling and optimization in MATLAB, In 2004 IEEE International Conference on Robotics and Automation (IEEE Cat. No.04CH37508), 284-289.

\bibitem{lo1}
J. L\"ofberg, {\em Pre- and Post-Processing Sum-of-Squares Programs in Practice}, IEEE Transactions on Automatic Control, 54(5)(2009):1007-1011.

\bibitem{toms17}
V. Magron, G. Constantinides, and A. Donaldson.
\newblock Certified roundoff error bounds using semidefinite programming.
\newblock {\em ACM Trans. Math. Software}, 43(4):Art. 34, 31, 2017.

\bibitem{toms18}
V. Magron.
\newblock {Interval Enclosures of Upper Bounds of Roundoff Errors Using
  Semidefinite Programming}.
\newblock {\em ACM Trans. Math. Softw.}, 44(4):41:1--41:18, June 2018.

\bibitem{maj}
A. Majumdar, A. A. Ahmadi and R. Tedrake, {\em Control and verification of high-dimensional systems with DSOS and SDSOS programming}, In 53rd IEEE Conference on Decision and Control, 2014:394-401.

\bibitem{ma}
A. Marandi, E. D. Klerk, and J. Dahl, {\em Solving sparse polynomial optimization problems with chordal structure using the sparse bounded-degree sum-of-squares hierarchy}, Discrete Applied Mathematics, 2017.

\bibitem{me}
A. Megretski, Systems polynomial optimization tools (SPOT), 2010. Available at https://github.com/spot-toolbox/spotless.

\bibitem{mi}
J. Miller, Y. Zheng, M. Sznaier, and A. Papachristodoulou, {\em Decomposed Structured Subsets for Semidefinite and Sum-of-Squares Optimization},
{\em preprint   \href{http://arxiv.org/abs/1911.12859}{arXiv:1911.12859}}, 2019.

\bibitem{mosek}
MOSEK ApS,
\newblock The MOSEK optimization toolbox. Version 8.1.,
\newblock \href{http://docs.mosek.com/8.1/toolbox/index.html}{MOSEK manual}, 2017.
 
\bibitem{nakata2003exploiting}
K. Nakata, K. Fujisawa, M. Fukuda, M. Kojima, and K.
  Murota,
\newblock Exploiting sparsity in semidefinite programming via matrix
  completion. {II}. {I}mplementation and numerical results,
\newblock {\em Math. Program.}, 95(2, Ser. B):303--327, 2003.

\bibitem{pa}
P. A. Parrilo, {\em Structured semidefinite programs and semialgebraic geometry methods in robustness and optimization}, Ph.D. Thesis, California Institute of Technology, 2000.

\bibitem{pe}
F. Permenter, P. A. Parrilo, {\em Basis selection for SOS programs via facial reduction and polyhedral approximations}, Decision and Control, IEEE, 2014:6615-6620.

\bibitem{pe1}
F. Permenter, P. A. Parrilo, {\em Finding sparse, equivalent SDPs using minimal coordinate projections}, In 54th IEEE Conference on Decision and Control, CDC 2015, Osaka, Japan,
December 15-18, 2015:7274-7279.

\bibitem{pu}
M. Putinar, {\em Positive polynomials on compact semialgebraic sets}, Indiana Univ. Math. J., 42(1993):969-984.

\bibitem{re}
B. Reznick, {\em Extremal PSD forms with few terms}, Duke Math. J., 45(1978):363-374.

\bibitem{Riener13}
C. Riener, T. Theobald, L.~J. Andr\'{e}n, and J.-B. Lasserre.
\newblock Exploiting symmetries in {SDP}-relaxations for polynomial
  optimization.
\newblock {\em Math. Oper. Res.}, 38(1):122--141, 2013.

\bibitem{Tacchi19}
M. Tacchi, T. Weisser, J.-B. Lasserre, and D. Henrion.
\newblock Exploiting sparsity for semi-algebraic set volume computation.
\newblock {\em preprint   \href{http://arxiv.org/abs/1902.02976}{arXiv:1902.02976}}, 2019.

\bibitem{sedumi}
J. F. Sturm,  
\newblock Using SeDuMi 1.02, a MATLAB toolbox for optimization over symmetric cones,
\newblock {\em Optimization methods and software}, 11(1-4)(1999): 625-653.

\bibitem{wang}
J. Wang, H. Li and B. Xia, {\em A New Sparse SOS Decomposition Algorithm Based on Term Sparsity}, in Proceedings of the 2019 on International Symposium on Symbolic and Algebraic Computation, ACM, 2019:347-354.

\bibitem{waki}
H. Waki, S. Kim, M. Kojima, and M. Muramatsu, {\em Sums of squares and semidefinite program relaxations for polynomial optimization problems with structured sparsity}, SIAM Journal on Optimization, 17(1)(2016):218-242.

\bibitem{WakiKKMS08}
H. Waki, S. Kim, M. Kojima, M. Muramatsu, and H. Sugimoto.
\newblock Algorithm 883: sparse{POP}---a sparse semidefinite programming relaxation of polynomial optimization problems.
\newblock {\em ACM Trans. Math. Software}, 35(2):Art. 15, 13, 2009.

\bibitem{chordaltssos}
J. Wang, V. Magron, and J.-B. Lasserre, {\em  Chordal-TSSOS: a moment-SOS hierarchy that exploits term sparsity with chordal extension}.
\newblock {\em preprint   \href{http://arxiv.org/abs/2003.03210}{arXiv:2003.03210}}, 2020.

\bibitem{cstssos}
J. Wang, V. Magron, J.-B. Lasserre  and N. H. A. Mai, {\em   CS-TSSOS: Correlative and term sparsity for large-scale polynomial optimization}.
\newblock {\em preprint   \href{http://arxiv.org/abs/2005.02828}{arXiv:2005.02828}}, 2020.

\bibitem{we}
T. Weisser, J. B. Lasserre, and K. C. Toh, {\em Sparse-BSOS: a bounded degree SOS hierarchy for large scale polynomial optimization with sparsity}, Mathematical Programming Computation, 10(1)(2018):1-32.

\end{thebibliography}
